\documentclass[a4paper,10pt]{article}
\usepackage[T1]{fontenc}
\usepackage{lmodern}
\usepackage{geometry}
\usepackage{amsthm}
\usepackage{amssymb}
\usepackage{amsmath}
\usepackage{xcolor}
\newtheorem{thm}{Theorem}[section]
\newtheorem{conjecture}[thm]{Conjecture}
\newtheorem{lemma}[thm]{Lemma}
\newtheorem{definition}[thm]{Definition}
\usepackage{cite}
\usepackage{tikz}
\usetikzlibrary{arrows.meta,calc,decorations.pathreplacing,patterns,positioning,shapes.geometric}

\usepackage[
  colorlinks=true,
  anchorcolor=blue,
  filecolor=blue,
  linkcolor=red,
  urlcolor=blue,
  citecolor=blue,
  hypertexnames=false,
  pdftitle={Longest cycles intersect linearly in highly connected graphs},
  pdfauthor={Jie Ma, Bo Ning, and Ziyuan Zhao}
]{hyperref}
\begin{document}

\title{Longest cycles intersect linearly in highly connected graphs}

\author{Jie Ma\thanks{School of Mathematical Sciences, University of Science and Technology of China, Hefei, Anhui 230026, China, and Yau Mathematical Sciences Center, Tsinghua University, Beijing 100084, China.
Research supported by National Key Research and Development Program of China 2023YFA1010201 and National Natural Science Foundation of China grant 12125106. Email: {\tt jiema@ustc.edu.cn}.}\and Bo Ning\thanks{College of Computer Science, Nankai University, Tianjin 300350, P.R. China.
Partially supported by the National Natural Science
Foundation of China (No. 12371350) and Fundamental Research Funds for the Central Universities, Nankai University (No. 63243151). Email: {\tt bo.ning@nankai.edu.cn}. }
\and Ziyuan Zhao\thanks{School of Mathematical Sciences, University of Science and Technology of China, Hefei, Anhui 230026, China. Research supported by Innovation Program for Quantum Science and Technology 2021ZD0302902. Email: {\tt zyzhao2024@mail.ustc.edu.cn}.}
}
\date{}
\maketitle
\begin{abstract}
A longstanding conjecture attributed to Smith (1984) asserts that for every $k\ge2$, any two longest cycles in a $k$-connected graph share at least $k$ vertices. In this paper, we prove the first linear lower bound, showing that any two longest cycles in a $k$-connected graph share at least $k/600$ vertices. Departing from previous Tur\'an-type extremal arguments, we develop a novel structural approach that also yields applications to related problems on longest cycles and paths.
\end{abstract}

\section{Introduction}
Cycles, and in particular longest cycles, are among the central objects of study in graph theory. A longstanding line of research concerns the structural properties of longest cycles in various graph families. In this paper, we study one of the classical problems in this area, often attributed to Smith (see Gr{\"o}tschel~\cite{grotschel1984intersections}).

\begin{conjecture}[Smith's conjecture]\label{conj:smith} 
For every integer $k\geq 2$, any two longest cycles in a $k$-connected graph share at least $k$ vertices.
\end{conjecture}

\noindent This conjecture was repeatedly mentioned in the influential surveys of Bondy~\cite{bondy1996basic,bondy2014beautiful} and listed as one of the 100 unsolved problems in graph theory in the textbook~\cite{bondy2008}.
The bound in Conjecture~\ref{conj:smith} is sharp, by considering the complete bipartite graph $K_{k,n-k}$ for $n\geq 3k$. This graph has vertex-connectivity $k$, while its longest cycles have length $2k$; moreover, by choosing two such cycles that use disjoint sets of $k$ vertices from the larger part, we obtain two longest cycles whose intersection consists exactly of the $k$ vertices of the smaller part.

Smith's conjecture remains open, and was previously known only for $2\leq k\leq 8$.
The case $k = 2$ follows readily from Menger's theorem. 
A proof for the case $k=3$ can be found in Babai~\cite{babai1979long} (see Lemma~2 therein), where it was applied ingeniously to derive a lower bound for the length of a longest cycle in vertex-transitive graphs.
A related early result on Smith's conjecture is due to Gr{\"o}tschel~\cite{grotschel1984intersections}, who proved that, for any $1\leq t\leq 5$, if $G$ is a 2-connected graph with at least $t+1$ vertices and contains
two longest cycles whose intersection is a set $W$ of exactly $t$ vertices, then $G-W$ is disconnected. Gr{\"o}tschel~\cite{grotschel1984intersections} subsequently conjectured that the same statement should hold for $t\in \{6,7\}$, which was later confirmed by Stewart and Thompson~\cite{stewart1995intersections}. 
These results establish the conjecture for $k\leq 8$.
Other references include~\cite{shabbir2013intersecting,zamfirescu2001intersecting}.
More recently, the conjecture has also been verified for very dense graphs: Guti{\'e}rrez and Valqui~\cite{gutierrez2024two} proved it for $n$-vertex graphs when $k\ge (n+16)/7$.

There has been extensive work on general lower bounds for the minimum intersection of two longest cycles in $k$-connected graphs.
The first nontrivial general lower bound was proved by Burr and Zamfirescu (reported in~\cite{chen1998intersections}), who showed that any two longest cycles in a $k$-connected graph have at least $\sqrt{k}-1$ common vertices.
This was improved by G. Chen, Faudree, and Gould~\cite{chen1998intersections} in 1998 to $\Omega(k^{3/5})$, using a novel Tur\'an-type argument that proved fruitful in subsequent works.
In 2025, Groenland, Longbrake, Steiner, Turcotte, and Yepremyan~\cite{groenland2024longest} further improved this to $\Omega(k^{5/8})$. 
Subsequently, the first and third authors of the present paper~\cite{ma2025intersections} used a supersaturation-type new argument to obtain the bound $\Omega(k^{2/3})$.
Most recently, D. Chen~\cite{chen2026improved} improved the bound to $\Omega(k^{8/11})$ with the aid of computer search.

The main result of this paper is a proof of Conjecture~\ref{conj:smith} up to a constant factor.

\begin{thm}\label{thm:smith-linear}
For every $k\geq 2$, any two longest cycles in a $k$-connected graph share at least $k/600$ vertices.
\end{thm}

We emphasize that the proof of Theorem~\ref{thm:smith-linear} differs from the extremal techniques (Tur\'an-type and supersaturation-type arguments) employed in the previous works~\cite{chen1998intersections,groenland2024longest,ma2025intersections,chen2026improved}. 
Instead, our proof relies on a completely structural approach, which includes several new developments for related problems.
In particular, we establish and make use of a strengthening of a theorem of Bondy and Locke (Theorem~\ref{lem:amplifier}) and a higher-connectivity extension of Dirac's theorem (Theorem~\ref{thm:small-circ}), both of which might be of independent interest.
We refer the reader to Section~\ref{sec:sketch} for a proof sketch of Theorem~\ref{thm:smith-linear}.

The rest of the paper is organized as follows.
In Section~\ref{sec:2}, we provide some preliminaries.
In Section~\ref{sec:sketch}, we give a heuristic outline of the proof of Theorem~\ref{thm:smith-linear}.
In Section~\ref{sec:amplifier}, we prove a strengthening of the Bondy--Locke theorem (Theorem~\ref{lem:amplifier}) through a new concept, called an ``amplifier'', which shows that, in any graph satisfying a local 3-connectivity condition, the lengths of a longest cycle and a longest path are of the same order.
In Section~\ref{sec:3}, we prove a structural decomposition result that ensures that the subgraphs under consideration behaves like 2-connected graphs.
In Section~\ref{sec:4}, we prove an edge-covering result using amplifiers and use it to prove Theorem~\ref{thm:smith-linear} under the assumption that the length of longest cycles is relatively large. 
In Section~\ref{sec:5}, we provide a self-contained proof of a higher-connectivity extension of Dirac's theorem (Theorem~\ref{thm:small-circ}).
In Section~\ref{sec:complete}, we complete the proof of Theorem~\ref{thm:smith-linear}.
In the final section, we conclude with several applications of our approach and discuss some related open problems.

\section{Preliminaries}\label{sec:2}
We use standard graph-theoretic notation as in \cite{diestel2024graph}.
Unless stated otherwise, all graphs are simple and finite.
For an integer $r$, let $[r]:=\{1,2,\ldots,r\}$.

Let $G$ be a graph with vertex set $V(G)$ and edge set $E(G)$. We use $\delta(G)$ to denote the minimum degree of $G$.
For either a subgraph or a vertex set $A$ of $G$, let $G-A$ denote the graph obtained from $G$ by deleting the vertices of $A$.
When $A=\{v\}$, we abbreviate $G-\{v\}$ by $G-v$.
We write $H\subseteq G$ when $H$ is a subgraph of $G$.
For a collection of graphs $\mathcal{H}$, define
$V(\mathcal{H})=\bigcup_{H\in\mathcal{H}}V(H)$ and
$E(\mathcal{H})=\bigcup_{H\in\mathcal{H}}E(H)$.

Let $A,B\subseteq V(G)$ be two non-empty sets, which are not necessarily distinct or disjoint. 
A path $P=x_1\cdots x_t$ in $G$ is called an \textit{$(A,B)$-path} if $x_1\in A$, $x_t\in B$, and $V(P)\setminus\{x_1,x_t\}$ is disjoint from $A\cup B$.
We write $(a,B)$-path for a $(\{a\},B)$-path.
If $A$ is a subgraph, we use $(A,B)$-path to denote a $(V(A),B)$-path.
We use the same convention in the case $B$ is a subgraph.
For a vertex set $S\subseteq V(G)$, we say $S$ \textit{separates $A$ and $B$} in $G$ if every $(A,B)$-path intersects $S$.
If $A$ is a subgraph of $G$, saying that $S$ separates $A$ and $B$
means that $S$ separates $V(A)$ and $B$.
A set $S\subseteq V(G)$ is a \emph{vertex cut}
of $G$ if $G-S$ is disconnected. 
If $\{v\}$ is a cut, then $v$ is a \textit{cut-vertex}.
A \textit{block} is a
maximal connected induced subgraph without a cut-vertex.
Hence, each block is either an isolated vertex, an edge together with
its endpoints, or a maximal 2-connected subgraph.

The length of a path or cycle is its number of edges; we denote the
lengths of $P$ and $C$ by $|P|$ and $|C|$, respectively.
We call a path $P$ \textit{trivial} if $|P|=0$, and \emph{nontrivial} otherwise.
For $a,b\in V(G)$, we write $\operatorname{dist}_G(a,b)$ for the minimum length of all $(a,b)$-paths in $G$. 
We denote by $c(G)$ the \textit{circumference} (that is, the
length of a longest cycle) of $G$ and by $p(G)$ the length of a longest
path in $G$.
For the empty graph $\emptyset$, we set $p(\emptyset):=0$.

Let $P$ be a path with endpoints $a,b$, and fix the orientation on $P$ from $a$ to $b$. For
$u,v\in V(P)$, write $u\leq_P v$ if $a,u,v,b$ occur on $P$ (with repetition permitted) along this orientation.
The relations $<_P$, $\geq_P$, and $>_P$ are defined similarly.
Two paths are \textit{disjoint} if their vertex sets are disjoint,
and \textit{internally disjoint} if every common vertex is an endpoint of at least one of the two paths.

For a tree $T$ and vertices $u,v \in V(T)$, we denote by $T[u,v]$ the unique $(u,v)$-subpath of $T$.
Define
$T[u,v):=T[u,v]-v$, $T(u,v]:=T[u,v]-u$, and
$T(u,v):=T[u,v]-\{u,v\}$.

We use the following form of Menger's theorem; see
\cite[Theorem~3.3.1]{diestel2024graph}.
\begin{thm}
Let $A$ and $B$ be disjoint vertex sets in a graph $G$, and let $t$ be a positive integer.
Then $G$ contains $t$ disjoint $(A,B)$-paths if and
only if no vertex set of size less than $t$ separates $A$ and $B$ in $G$.
\end{thm}

We also use the following classical theorem of Dirac~\cite{dirac1952some}.
\begin{thm}[Dirac]\label{thm:dirac}
Every $2$-connected graph $G$ satisfies
$c(G)\geq\min\{2\delta(G),|V(G)|\}$.
In particular, for $k\geq 2$, every $k$-connected graph $G$ satisfies $c(G)\geq k$.
\end{thm}

\section{Proof sketch for Theorem~\ref{thm:smith-linear}}\label{sec:sketch}
In this section, we give a heuristic outline of the proof of Theorem~\ref{thm:smith-linear}.
Let $X$ and $Y$ be two longest cycles in a $k$-connected graph $G$. 
Assume that $|V(X)\cap V(Y)|=m$.
The goal is to show that $m=\Omega(k)$.

It is clear that we have $m\geq 2$ (as $G$ is $2$-connected) and by Theorem~\ref{thm:dirac}, $c(G)=|X|=|Y|\geq k$.
Let $G_X:=G-X$, and let $\mathcal Y$ be the collection of nontrivial maximal paths of $Y-X$. 
If $\mathcal Y$ is empty, then $|Y|\leq2m$, and hence $m\geq |Y|/2\geq k/2$, as desired. 
We may therefore assume that $\mathcal Y$ is nonempty. 
Since deleting the $m$ vertices of $V(X)\cap V(Y)$ removes at most $2m$ edges of $Y$, we have $|E(\mathcal Y)|\geq c(G)-2m=\Omega(c(G))$.\footnote{For instance, if $c(G)-2m\leq c(G)/2$, then $m\geq c(G)/4\geq k/4$, as desired.}

To illustrate the proof of Theorem~\ref{thm:smith-linear}, we first consider the most ideal scenario that 
\begin{align}\label{eq:model-case}
    \text{every component of $G_X$ is 3-connected and has at least $k$ vertices.}
\end{align}
This model case serves as our proof motivation and exhibits the general proof mechanisms. 
Let $H$ be a component of $G_X$ containing a longest path in $\mathcal{Y}$.
Then $p(H)\geq |E(\mathcal{Y})|/m\geq \Omega(c(G)/m)$.
A celebrated result of Bondy and Locke \cite{bondy1981relative} states that, in every $3$-connected graph, the lengths of a longest cycle and a longest path are of the same magnitude.
Fix a longest cycle $C$ in $H$. By~\eqref{eq:model-case}, $H$ is $3$-connected and thus we have $|C|=c(H)=\Omega(p(H))=\Omega(c(G)/m)$.
Note that $|V(X)|\geq k$, and $|V(H)|\geq k$ from \eqref{eq:model-case}.
Since $G$ is $k$-connected, by Menger's theorem, there exist $k$ disjoint $(X,H)$-paths in $G$.
Let these paths intersect $X$ at the vertices $u_1, \dots, u_k$ in cyclic order. 
Since $H$ is 2-connected (in fact, 3-connected by~\eqref{eq:model-case}), for each $i \in [k]$, $G$ contains two disjoint $(X,C)$-paths with endpoints $u_i$ and $u_{i+1}$ in $X$. 
By the maximality of $|X|$, a standard cycle-switching argument shows that $\operatorname{dist}_X(u_i,u_{i+1})\geq |C|/2=\Omega(c(G)/m)$, where
$u_{k+1}:=u_1$. 
Summing over all $i\in[k]$ gives
\[
c(G) = |X| \geq \sum_{i\in [k]}\operatorname{dist}_X(u_i, u_{i+1})=k \cdot\Omega(c(G)/m),
\]
which simplifies to $m = \Omega(k)$, as required.

The applications of~\eqref{eq:model-case} in the preceding proof reveal three obvious obstructions to extending this model case to the general case, as follows. Let $H$ denote a component of $G_X$ containing a path in $\mathcal{Y}$ of length $\Omega(c(G)/m)$.
\begin{itemize}
    \item[(a).] $H$ may not contain a cycle $C$ of length $\Omega(p(H))$ if $H$ is not 3-connected.
    \item[(b).] The required two disjoint $(X,C)$-paths may not exist if $H$ is not 2-connected.
    \item[(c).] There cannot be $k$ disjoint $(X,H)$-paths if $|V(H)|<k$.
\end{itemize}
Our strategy for addressing these obstructions consists of three key parts, each developed in one of Sections~\ref{sec:3},~\ref{sec:4}, and~\ref{sec:5}. 
First, to restore the 2-connectivity needed in (b), we prove a decomposition result (Theorem~\ref{thm:del1cut}) in Section~\ref{sec:3}, which, roughly speaking, decomposes the components of $G_X$ into $O(m)$ subgraphs that behave like $2$-connected graphs, while decomposing the edges in $Y-X$ into $O(m)$ subpaths. 
This resolves (b).
Second, based on this decomposition result, in Section~\ref{sec:4} we further ``refine'' these $O(m)$ subpaths so that most edges of $Y-X$ are covered, and each resulting subpath has a cycle containing a positive fraction of its edges. 
This relies heavily on a strengthening of the Bondy--Locke theorem that we establish in Section~\ref{sec:amplifier}, under a much weaker connectivity assumption than 3-connectivity.
We may thus further assume that (a) is resolved. 
It remains to consider (c). If the component $H$ chosen in the preceding proof has fewer than $k$ vertices, then $\Omega(c(G)/m) \leq p(H)\leq k-1$,
implying that $c(G)=O(mk)$.
In this range, in Section~\ref{sec:5} we prove an extension of Dirac's theorem in highly connected graphs (Theorem~\ref{thm:small-circ}) and use it to obtain a lower bound on $c(G)$, leading to the final contradiction.
Together, these three parts complete the proof of Theorem~\ref{thm:smith-linear}.

\section{Beyond the Bondy-Locke theorem -- amplifiers}\label{sec:amplifier}
Recall that $c(G)$ and $p(G)$ denote the lengths of a longest cycle and a
longest path in a graph $G$, respectively. 

A celebrated theorem of Bondy and
Locke~\cite[Theorem~2]{bondy1981relative} states that every 3-connected graph
$G$ satisfies $c(G)\ge 2p(G)/5$. 
Their proof actually gives a stronger assertion:
for every path $P$ in a 3-connected graph $G$, there is a cycle in $G$ containing at least $2|P|/5$ edges of $P$. This motivates the following definition.

\begin{definition}[$(H,\alpha)$-amplifier]\label{def:amplifier}
Let $P$ be a path in a graph $H$, and let $\alpha>0$. We call $P$ an
\textbf{$(H,\alpha)$-amplifier} if some cycle $C\subseteq H$ satisfies
$|E(C)\cap E(P)|\geq \alpha|P|$.
\end{definition}

Using this notation, the Bondy-Locke theorem says that every path in a 3-connected graph $G$ is a $(G,2/5)$-amplifier.
For distinct $a,b\in V(G)$, let $\kappa_G(a,b)$ denote the maximum number of pairwise internally disjoint $(a,b)$-paths in $G$. By Menger's theorem, the connectivity of $G$ satisfies $\kappa(G)=\min_{a,b\in V(G)}\kappa_G(a,b)$. We show, with a short proof, that the weaker local condition $\kappa_G(a,b)\geq 3$ suffices to ensure that every $(a,b)$-path in any graph $G$ is a $(G,1/6)$-amplifier.

\begin{thm}\label{lem:amplifier}
Let $H$ be a multigraph with two given vertices $a$ and $b$.\footnote{In the remainder of this section, the term \textit{multigraph} refers to a graph that may contain parallel edges and loops, while a \textit{graph} always refers to a simple graph.
Parallel edges and loops are treated as cycles of length two and one, respectively.} Suppose that $H$ contains three internally disjoint $(a,b)$-paths. Then for every $(a,b)$-path $P$ in $H$, there is a cycle $C$ in $H$ passing through $a$ and $b$ such that $|E(C) \cap E(P)| \geq |P|/6$.
\end{thm}

\begin{proof}
Fix $P$, and choose three internally disjoint $(a,b)$-paths
$P_1,P_2,P_3$ minimizing the size of the edge set  
$$F=F(P_1,P_2,P_3):=\bigl(E(P_1)\cup E(P_2)\cup E(P_3)\bigr)\setminus E(P).$$
Let $<_i$ denote the vertex ordering on $P_i$ from $a$ to $b$. 
Let $\mathcal{L}$ be the collection of maximal nontrivial subpaths $L$
of $P$ satisfying both of the following conditions: 
\begin{itemize}
    \item[(1)] $E(L)\subseteq E(P)\setminus
\bigl(E(P_1)\cup E(P_2)\cup E(P_3)\bigr)$ for every $L\in \mathcal{L}$;
    \item[(2)] all internal
vertices of $L$ lie outside $\bigcup_{i\in[3]}V(P_i)$.\footnote{In Figure~1(c), condition~(2) ensures that the three paths $L_2,L_3,L_4$ belong to $\mathcal L$, rather than the single path $L_2\cup L_3\cup L_4$.}
\end{itemize}
Thus, every path $L\in\mathcal L$ has both endpoints in $\bigcup_{i\in[3]}V(P_i)$, and the paths in $\mathcal L$ are pairwise internally disjoint. However, the paths in $\mathcal L$ may have common endpoints; see the red paths in Figure~\ref{fig:amplifier}(c) for an example.

We first claim that no path $L\in\mathcal L$ has both endpoints on the same
path among $P_1,P_2,P_3$. Suppose for contradiction that some $L\in\mathcal L$ has both endpoints $x<_{1}y$ on $P_1$. In Figure~\ref{fig:amplifier}(a), the black curves are $P_1,P_2,P_3$, the red curve is $L$, and the right-hand diagram shows three internally disjoint paths $P_1',P_2,P_3$, where $P_1'$ is obtained from $P_1$ by replacing $P_1[x,y]$ with $L$. 
This replacement adds no edge to $F$, and removes
every edge of the non-empty set $E(P_1[x,y])\setminus E(P)$ from $F$.
Indeed, if $P_1[x,y]$ contained only edges of $P$, then it would be the
unique $(x,y)$-subpath of $P$ and would coincide with $L$, a contradiction to the definition of $L$. Hence, the replacement strictly
decreases $|F|$, a contradiction to the minimality of $|F|$. This proves the claim.

We next claim that, for each $\{i,j\}\subseteq[3]$, the paths in
$\mathcal L$ connecting $P_i$ and $P_j$ can be indexed so that their endpoints on $P_i$ and $P_j$
occur in nondecreasing order in $<_i$ and $<_j$ (with repetitions allowed), respectively.
It suffices to consider $P_1$ and $P_2$. 
Suppose for contradiction that there are
$L_1,L_2\in \mathcal L$ such that $L_s$ has endpoints
$x_s\in V(P_1)$ and $y_s\in V(P_2)$ for $s\in[2]$, with
$x_1<_{1}x_2$ and $y_2<_{2}y_1$. As shown in
Figure~\ref{fig:amplifier}(b), replacing $P_1,P_2$
by
\[
 P_1':=P_1[a,x_1]\cup L_1\cup P_2[y_1,b]
 \quad\text{and}\quad
 P_2':=P_2[a,y_2]\cup L_2\cup P_1[x_2,b]
\]
again gives three internally disjoint $(a,b)$-paths. At least one of
$P_1[x_1,x_2]$ and $P_2[y_2,y_1]$ contains an edge outside $P$;
otherwise
$P_1[x_1,x_2]\cup L_2\cup P_2[y_2,y_1]\cup L_1$ would be a cycle
whose edges all lie in $P$, a contradiction. Thus this replacement also strictly
decreases $|F|$, a contradiction.
This proves the claim.

Fix $\{i,j\}\subseteq[3]$. By the preceding two claims, the paths in
$\mathcal L$ connecting $P_i$ and $P_j$ can be indexed as
$L_1,\ldots,L_r$ so that their endpoints occur in nondecreasing index
order on both $P_i$ and $P_j$. See Figure~\ref{fig:amplifier}(c), where the paths $L_s$ are colored red. Each vertex is an endpoint of at most
two paths in $\mathcal L$, and two paths with a common endpoint must have consecutive indices. Hence the odd-indexed paths are pairwise
disjoint, and so are the even-indexed paths.
For each parity, alternating between $P_i$ and $P_j$ along the paths with indices of this given parity
gives the two zig-zag $(a,b)$-paths, as shown in the right figure in Figure~\ref{fig:amplifier}(c). Together
these four paths cover
every edge of $P_i\cup P_j$ and every path in 
$\mathcal L$ that connects $P_i$ and $P_j$ exactly twice.

Applying this for the three pairs $\{1,2\},\{1,3\},\{2,3\}$ gives twelve
$(a,b)$-paths. Every
edge in $\bigcup_iE(P_i)$ is counted four times, and every edge in
$E(\mathcal L)$ twice. Since
$E(P)\subseteq\bigcup_{i=1}^3E(P_i)\cup E(\mathcal L)$, one of the
twelve paths, say $Q$, contains at least $2|P|/12=|P|/6$ edges of $P$. This path $Q$ uses
only two of the paths $P_1,P_2,P_3$, say $P_1$ and $P_2$, and paths in $\mathcal L$ connecting
them. Then $Q$ and $P_3$ are
internally disjoint, so $Q\cup P_3$ gives the required cycle $C$. This proves Theorem~\ref{lem:amplifier}.
\end{proof}

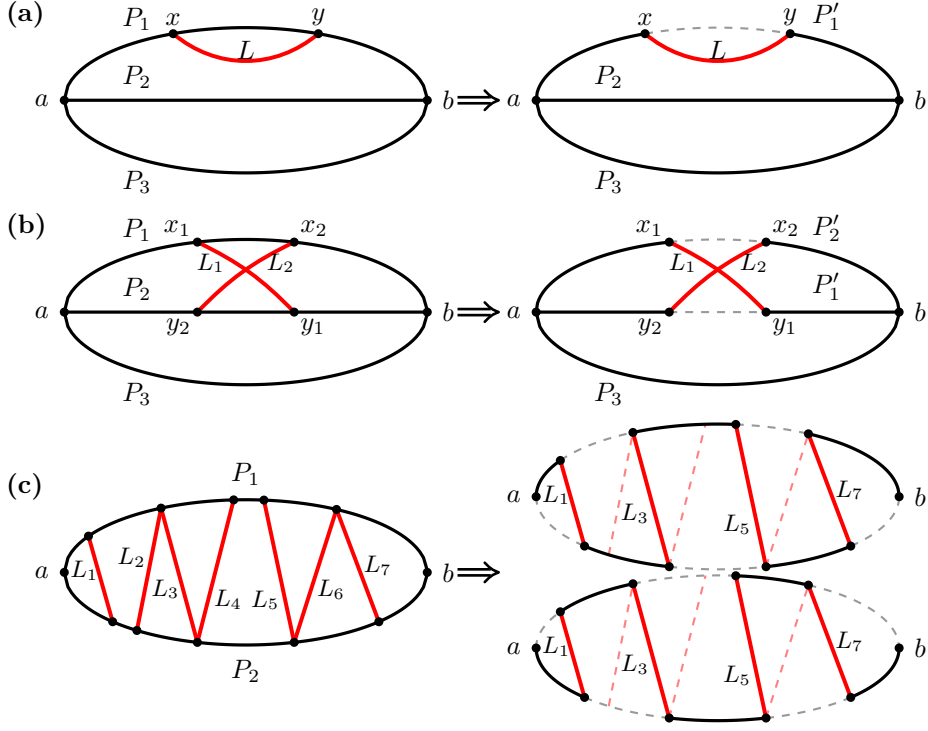
\begin{figure}[htb]
    \centering
    
\tikzset{every picture/.style={line width=0.75pt}} 
\tikzset{dot/.style={fill=black,circle,inner sep=1.0pt}} 

\begin{tikzpicture}[
    dot/.style = {circle, fill, inner sep=1.2pt},
    p_path/.style = {black, line width=1.2pt},
    p_dash/.style = {black, thick, dashed, opacity=0.4},
    l_path/.style = {red, line width=1.5pt},
    l_dash/.style = {red, thick, dashed, opacity=0.5},
    samples=100,
    scale=0.8,
    declare function={
        P(\x) = 1.2 * sqrt(max(0, 1 - (\x/3)^2));
    }
]

\newcommand{\myarrow}[2]{
    \begin{scope}[shift={(#1, #2)}]
        \draw[black, line width=1pt] (0, 0.05) -- (0.55, 0.05);
        \draw[black, line width=1pt] (0, -0.05) -- (0.55, -0.05);
        
        \draw[black, line width=1pt, line cap=round, line join=round]
            (0.45, 0.16) .. controls (0.56, 0.05) .. (0.68, 0)
                         .. controls (0.56, -0.05) .. (0.45, -0.16);
    \end{scope}
}

\begin{scope}
    \node[anchor=north west, font=\bfseries,xshift=-5mm] at (-3.5, 1.8) {(a)};

    \newcommand{\defineClaimOnePoints}{
        \coordinate (Qx) at (-1.2, {P(-1.2)});
        \coordinate (Qy) at (1.2, {P(1.2)});
    }

    \begin{scope}
        \defineClaimOnePoints
        \node[dot, label=left:$a$] at (-3, 0) {};
        \node[dot, label=right:$b$] at (3, 0) {};
        
        \draw[p_path] plot[domain=-3:3] (\x, {P(\x)});
        \draw[p_path] (-3,0) -- (3,0);
        \draw[p_path] plot[domain=-3:3] (\x, {-P(\x)});
        
        \node[above=1pt] at (-1.8, {P(-1.8)}) {$P_1$};
        \node[above=1pt] at (-1.8, 0) {$P_2$};
        \node[below=1pt] at (-1.8, {-P(-1.8)}) {$P_3$};
        
        \draw[l_path] (Qx) to[bend right=40] node[pos=0.5, anchor=south, inner sep=3pt, black, yshift=-1mm] {$L$} (Qy);
        
        \node[dot] at (Qx) {}; \node[above, inner sep=3pt] at (Qx) {$x$};
        \node[dot] at (Qy) {}; \node[above, inner sep=3pt] at (Qy) {$y$};
    \end{scope}

    \myarrow{3.5}{0}

    \begin{scope}[shift={(7.8, 0)}]
        \defineClaimOnePoints
        \node[dot, label=left:$a$] at (-3, 0) {};
        \node[dot, label=right:$b$] at (3, 0) {};
        
        \draw[p_path] (-3,0) -- (3,0);
        \draw[p_path] plot[domain=-3:3] (\x, {-P(\x)});
        
        \draw[p_dash] plot[domain=-1.2:1.2] (\x, {P(\x)}); 
        
        \draw[p_path] plot[domain=-3:-1.2] (\x, {P(\x)});
        \draw[l_path] (Qx) to[bend right=40] node[pos=0.5, anchor=south, inner sep=3pt, black,yshift=-1mm] {$L$} (Qy);
        \draw[p_path] plot[domain=1.2:3] (\x, {P(\x)});
        
        \node[above=1pt] at (1.8, {P(1.8)}) {$P_1'$};
        \node[above=1pt] at (-1.8, 0) {$P_2$};
        \node[below=1pt] at (-1.8, {-P(-1.8)}) {$P_3$};
        
        \node[dot] at (Qx) {}; \node[above, inner sep=3pt] at (Qx) {$x$};
        \node[dot] at (Qy) {}; \node[above, inner sep=3pt] at (Qy) {$y$};
    \end{scope}
\end{scope}

\begin{scope}[shift={(0, -3.5)}]
    \node[anchor=north west, font=\bfseries,xshift=-5mm] at (-3.5, 1.8) {(b)};

    \newcommand{\defineClaimTwoPoints}{
        \coordinate (x1) at (-0.8, {P(-0.8)});
        \coordinate (y1) at (0.8, 0);
        \coordinate (x2) at (0.8, {P(0.8)});
        \coordinate (y2) at (-0.8, 0);
    }

    \begin{scope}
        \defineClaimTwoPoints
        \node[dot, label=left:$a$] at (-3, 0) {};
        \node[dot, label=right:$b$] at (3, 0) {};
        
        \draw[p_path] plot[domain=-3:3] (\x, {P(\x)});
        \draw[p_path] (-3,0) -- (3,0);
        \draw[p_path] plot[domain=-3:3] (\x, {-P(\x)});
        
        \node[above=1pt] at (-1.8, {P(-1.8)}) {$P_1$};
        \node[above=1pt] at (-1.8, 0) {$P_2$};
        \node[below=1pt] at (-1.8, {-P(-1.8)}) {$P_3$};
        
        \draw[l_path] (x1) to[bend left=10] node[pos=0.3, anchor=east, font=\small, black, inner sep=1.5pt, yshift=-0.4mm] {$L_1$} (y1);
        \draw[l_path] (x2) to[bend right=10] node[pos=0.3, anchor=west, font=\small, black, inner sep=1.5pt, yshift=-0.4mm] {$L_2$} (y2);
        
        \foreach \p in {x1, y1, x2, y2} \node[dot] at (\p) {};
        \node[above left, inner sep=2pt] at (x1) {$x_1$};
        \node[below right, inner sep=2pt] at (y1) {$y_1$};
        \node[above right, inner sep=2pt] at (x2) {$x_2$};
        \node[below left, inner sep=2pt] at (y2) {$y_2$};
    \end{scope}

    \myarrow{3.5}{0}

    \begin{scope}[shift={(7.8, 0)}]
        \defineClaimTwoPoints
        \node[dot, label=left:$a$] at (-3, 0) {};
        \node[dot, label=right:$b$] at (3, 0) {};
        
        \draw[p_path] plot[domain=-3:3] (\x, {-P(\x)});
        \node[below=1pt] at (-1.8, {-P(-1.8)}) {$P_3$};
        
        \draw[p_dash] plot[domain=-0.8:0.8] (\x, {P(\x)}); 
        \draw[p_dash] (-0.8, 0) -- (0.8, 0);               

        \draw[p_path] plot[domain=-3:-0.8] (\x, {P(\x)}); 
        \draw[l_path] (x1) to[bend left=10] node[pos=0.3, anchor=east, font=\small, black, inner sep=1.5pt, yshift=-0.4mm] {$L_1$} (y1);
        \draw[l_path] (x2) to[bend right=10] node[pos=0.3, anchor=west, font=\small, black, inner sep=1.5pt, yshift=-0.4mm, xshift=0.16mm] {$L_2$} (y2);
        
        \draw[p_path] (0.8, 0) -- (3, 0);                 
        \node[above=1pt] at (1.8, 0) {$P_1'$};
        
        \draw[p_path] (-3, 0) -- (-0.8, 0);               
        \draw[p_path] plot[domain=0.8:3] (\x, {P(\x)});   
        \node[above=1pt] at (1.8, {P(1.8)}) {$P_2'$};
        
        \foreach \p in {x1, y1, x2, y2} \node[dot] at (\p) {};
        \node[above left, inner sep=2pt] at (x1) {$x_1$};
        \node[below right, inner sep=2pt] at (y1) {$y_1$};
        \node[above right, inner sep=2pt] at (x2) {$x_2$};
        \node[below left, inner sep=2pt] at (y2) {$y_2$};
    \end{scope}
\end{scope}

\begin{scope}[shift={(0, -7.8)}]
    \node[anchor=north west, font=\bfseries,xshift=-5mm] at (-3.5, 1.8) {(c)};

    \newcommand{\defineZigZagPoints}{
        \coordinate (T1) at (-2.6, {P(-2.6)});
        \coordinate (B1) at (-2.2, {-P(-2.2)});
        
        \coordinate (B2) at (-1.8, {-P(-1.8)});
        \coordinate (T2) at (-1.4, {P(-1.4)});
        
        \coordinate (B3) at (-0.8, {-P(-0.8)});
        \coordinate (T3) at (-0.2, {P(-0.2)});
        
        \coordinate (T4) at (0.3,  {P(0.3)});
        \coordinate (B4) at (0.8,  {-P(0.8)});
        
        \coordinate (T5) at (1.5,  {P(1.5)});
        \coordinate (B5) at (2.2,  {-P(2.2)});
    }

    \begin{scope}
        \defineZigZagPoints
        \node[dot, label=left:$a$] at (-3, 0) {};
        \node[dot, label=right:$b$] at (3, 0) {};
        
        \draw[p_path] plot[domain=-3:3] (\x, {P(\x)});
        \draw[p_path] plot[domain=-3:3] (\x, {-P(\x)});
        
        \node[above=2pt] at (0, {P(0)}) {$P_1$};
        \node[below=2pt] at (0, {-P(0)}) {$P_2$};
        
        \draw[l_path] (T1) -- node[pos=0.4, anchor=east, font=\small, inner sep=1.5pt, black, xshift=0.8mm]{$L_1$} (B1); 
        \draw[l_path] (B2) -- node[pos=0.6, anchor=east, font=\small, inner sep=1.5pt, black]{$L_2$} (T2); 
        \draw[l_path] (T2) -- node[pos=0.6, anchor=east, font=\small, inner sep=1.5pt, black,xshift=0.3mm]{$L_3$} (B3); 
        \draw[l_path] (B3) -- node[pos=0.3, anchor=west, font=\small, inner sep=1.5pt, black]{$L_4$} (T3); 
        \draw[l_path] (T4) -- node[pos=0.7, anchor=east, font=\small, inner sep=1.5pt, black]{$L_5$} (B4); 
        \draw[l_path] (B4) -- node[pos=0.4, anchor=west, font=\small, inner sep=1.5pt, black]{$L_6$} (T5); 
        \draw[l_path] (T5) -- node[pos=0.5, anchor=west, font=\small, inner sep=1.5pt, black]{$L_7$} (B5);
        
        \foreach \p in {T1, B1, B2, T2, B3, T3, T4, B4, T5, B5} \node[dot] at (\p) {};
    \end{scope}

    \myarrow{3.5}{0}

    \begin{scope}[shift={(7.8, 0)}]
        
        \begin{scope}[shift={(0, 1.25)}] 
            \defineZigZagPoints
            \node[dot, label=left:$a$] at (-3, 0) {};
            \node[dot, label=right:$b$] at (3, 0) {};
            
            \draw[p_dash] plot[domain=-2.6:-1.4] (\x, {P(\x)}); 
            \draw[p_dash] plot[domain=0.3:1.5] (\x, {P(\x)});   
            \draw[p_dash] plot[domain=-3:-2.2] (\x, {-P(\x)}); 
            \draw[p_dash] plot[domain=-0.8:0.8] (\x, {-P(\x)}); 
            \draw[p_dash] plot[domain=2.2:3] (\x, {-P(\x)});   
            
            \draw[l_dash] (B2) -- (T2); 
            \draw[l_dash] (B3) -- (T3); 
            \draw[l_dash] (B4) -- (T5); 

            \draw[p_path] plot[domain=-3:-2.6] (\x, {P(\x)}); 
            \draw[l_path] (T1) -- node[pos=0.4, anchor=east, font=\small, inner sep=1.5pt, black, xshift=0.8mm]{$L_1$} (B1);                       
            \draw[p_path] plot[domain=-2.2:-0.8] (\x, {-P(\x)});
            \draw[l_path] (B3) -- node[pos=0.4, anchor=east, font=\small, inner sep=1.5pt, black]{$L_3$} (T2);                    
            \draw[p_path] plot[domain=-1.4:0.3] (\x, {P(\x)});  
            \draw[l_path] (T4) -- node[pos=0.7, anchor=east, font=\small, inner sep=1.5pt, black]{$L_5$} (B4);                       
            \draw[p_path] plot[domain=0.8:2.2] (\x, {-P(\x)});  
            \draw[l_path] (B5) -- node[pos=0.5, anchor=west, font=\small, inner sep=1.5pt, black]{$L_7$} (T5);                     
            \draw[p_path] plot[domain=1.5:3] (\x, {P(\x)});     
            
            \foreach \p in {T1, B1, B3, T2, T4, B4, B5, T5} \node[dot] at (\p) {};
        \end{scope}

        \begin{scope}[shift={(0, -1.25)}] 
            \defineZigZagPoints
            \node[dot, label=left:$a$] at (-3, 0) {};
            \node[dot, label=right:$b$] at (3, 0) {};
            
            \draw[p_dash] plot[domain=-3:-2.6] (\x, {P(\x)});   
            \draw[p_dash] plot[domain=-1.4:0.3] (\x, {P(\x)});  
            \draw[p_dash] plot[domain=1.5:3] (\x, {P(\x)});     
            \draw[p_dash] plot[domain=-2.2:-0.8] (\x, {-P(\x)}); 
            \draw[p_dash] plot[domain=0.8:2.2] (\x, {-P(\x)});  
            
            \draw[l_dash] (B2) -- (T2); 
            \draw[l_dash] (B3) -- (T3); 
            \draw[l_dash] (B4) -- (T5); 

            \draw[p_path] plot[domain=-3:-2.2] (\x, {-P(\x)}); 
            \draw[l_path] (B1) -- node[pos=0.6, anchor=east, font=\small, inner sep=1.5pt, black, xshift=0.8mm]{$L_1$} (T1);                     
            \draw[p_path] plot[domain=-2.6:-1.4] (\x, {P(\x)}); 
            \draw[l_path] (T2) -- node[pos=0.6, anchor=east, font=\small, inner sep=1.5pt, black]{$L_3$} (B3);                       
            \draw[p_path] plot[domain=-0.8:0.8] (\x, {-P(\x)}); 
            \draw[l_path] (B4) -- node[pos=0.3, anchor=east, font=\small, inner sep=1.5pt, black]{$L_5$} (T4);                      
            \draw[p_path] plot[domain=0.3:1.5] (\x, {P(\x)});   
            \draw[l_path] (T5) -- node[pos=0.5, anchor=west, font=\small, inner sep=1.5pt, black]{$L_7$} (B5);                       
            \draw[p_path] plot[domain=2.2:3] (\x, {-P(\x)});     
            
            \foreach \p in {T1, B1, T2, B3, T4, B4, T5, B5} \node[dot] at (\p) {};
        \end{scope}
    \end{scope}

\end{scope}

\end{tikzpicture}
    \caption{Two path-switching operations and the zig-zag construction in the proof of Theorem~\ref{lem:amplifier}.}
    \label{fig:amplifier}
\end{figure}

Using Theorem~\ref{lem:amplifier}, we derive the following lemma, which serves as a key tool in Section~\ref{sec:4} for identifying amplifiers.

\begin{lemma}\label{lem:local-amplifier}
Let $P$ be an $(a,b)$-path in a graph $H$, and let
$a<_P c<_P d<_P b$. Suppose that $P[a,c]$ and $P[d,b]$ each contain at
least three vertices and that, for every $v\in V(P[c,d])$, the graph
$H-v$ contains two disjoint $(P[a,v),P(v,b])$-paths. Then $P[c,d]$ is
an $(H,1/6)$-amplifier.
\end{lemma}

\begin{proof}
Let $A:=P[a,c]$ and $B:=P[d,b]$. We first claim that $H$ contains
three disjoint $(A,B)$-paths. If not, Menger's theorem gives a set
$S$ of at most two vertices separating $A$ from $B$. Since $P[c,d]$
is an $(A,B)$-path, the set $S$ meets $P[c,d]$.
Suppose first that $S\cap V(P[c,d])=\{x\}$. The hypothesis at $x$ gives
two disjoint $(P[a,x),P(x,b])$-paths in $H-x$, one of which avoids the
possible second vertex of $S$. The union of this path with
$P[c,x)\cup P(x,d]$ contains an $(A,B)$-path in $H-S$, a contradiction.
We may therefore write $S=\{x,y\}\subseteq V(P[c,d])$, where
$x<_P y$. Among the two disjoint $(P[a,x),P(x,b])$-paths provided by the hypothesis at $x$, one avoids $y$.
We denote this path by $L_x$. If its endpoint on $P(x,b]$ lies in
$P(y,b]$, then the union
$L_x\cup P[c,x)\cup P(y,b]$ contains an $(A,B)$-path in $H-S$, a
contradiction. Hence $L_x$ connects the subpaths $P[a,x)$ and
$P(x,y)$. Symmetrically, $H-S$ contains a path $L_y$ connecting $P(x,y)$ to
$P(y,b]$. Then the union
$P[a,x)\cup L_x\cup P(x,y)\cup L_y\cup P(y,b]$ contains an
$(A,B)$-path in $H-S$, a contradiction.
This proves the claim.

By the claim, there are three disjoint $(A,B)$-paths in $H$. Let $H'$
be the multigraph obtained from $H$ by contracting $A$ and $B$ to
vertices $a'$ and $b'$, respectively, and deleting the resulting loops.
The images of the three disjoint $(A,B)$-paths are internally disjoint $(a',b')$-paths in $H'$,
and the image $P'$ of $P$ is an $(a',b')$-path of length
$|P[c,d]|$. By Theorem~\ref{lem:amplifier}, there is a cycle $C'$ through
$a',b'$ with $|E(C')\cap E(P')|\geq |P[c,d]|/6$.
Replace $a'$ in $C'$ with the (possibly trivial) subpath of $A$ joining the original endpoints in $A$ of the two edges incident with $a'$. Apply the analogous replacement at $b'$ using a subpath of $B$. 
Since all vertices outside $P$ are preserved in this contraction, these replacements give a cycle $C\subseteq H$ with
$|E(C)\cap E(P[c,d])|=|E(C')\cap E(P')|\geq |P[c,d]|/6$.
Thus $P[c,d]$ is an
$(H,1/6)$-amplifier.
\end{proof}

\section{A bounded structural decomposition}\label{sec:3}
Throughout this section, let $X$ and $Y$ be two longest cycles in a graph
$G$ with $m:=|V(X)\cap V(Y)|\geq2$. Let $G_X:=G-X$, and let $\mathcal Y$ be the collection of nontrivial maximal paths of $Y-X$. 
Note that $|\mathcal Y|\leq m$ and the paths in $\mathcal Y$ are pairwise disjoint.

The goal of this section is to prove a decomposition result that partitions the edges of $Y-X$ into $O(m)$ paths, each contained in one of $O(m)$ induced subgraphs of $G_X$ that behave like $2$-connected graphs. 
The formal statement is given below.

\begin{thm}\label{thm:del1cut}
There exists a collection $\mathcal P$ of edge-disjoint nontrivial
subpaths of $Y-X$, a collection $\mathcal H$ of pairwise edge-disjoint
connected induced subgraphs of $G_X$, and a set
$\mathcal S\subseteq V(G_X)$ with the following properties.
\begin{itemize}
    \item[(1)] $|\mathcal P|\leq5m$, $|\mathcal H|\leq4m$, and
    $|\mathcal S|\leq4m$.
    \item[(2)] $E(\mathcal P)=E(Y-X)$, and every
    $P\in\mathcal P$ is contained in exactly one $H\in\mathcal H$.
    \item[(3)] Every $H\in \mathcal{H}$ contains at least one $P\in \mathcal{P}$, and the members of $\mathcal P$ contained in $H$ are pairwise vertex-disjoint. Moreover,
    for any distinct $H,H'\in\mathcal H$, the set $\mathcal S$
    separates $H$ and $H'$ in $G_X$.
    \item[(4)] Every $H\in\mathcal H$ is a union of blocks of $G_X$. Moreover, if $H\in\mathcal H$ contains at least two paths $P\in\mathcal P$, then $H$ is a block in $G_X$ and $|V(H)|\geq3$.
    \item[(5)] If $H\in\mathcal H$ contains at least two paths $P\in \mathcal{P}$, and $x_1,x_2\in V(X)$ are the endpoints on $X$ of two disjoint $(X,H)$-paths in $G$, then
    $\operatorname{dist}_X(x_1,x_2)\geq c(H)/2$.
\end{itemize}
\end{thm}

For each $H\in\mathcal H$, let
$\mathcal P_H:=\{P\in\mathcal P:P\subseteq H\}$. 
Note that by (2), $\mathcal P_H$ is non-empty.

\medskip

We begin with some basic properties of blocks that we will use later.
For a connected graph $F$, let $T(F)$ be its \textbf{block-cut tree}\footnote{The
fact that $T(F)$ is a tree is the standard block graph theorem; see,
for example, \cite[Section~3]{diestel2024graph}.} with bipartition
$(\mathcal B(F),\mathcal C(F))$, where $\mathcal B(F)$ is the collection
of blocks and $\mathcal C(F)$ is the set of cut-vertices in $F$. 
A block $B\in\mathcal B(F)$ and a cut-vertex $x\in\mathcal C(F)$ are adjacent in
$T(F)$ precisely when $x\in V(B)$.

\begin{lemma}\label{lem:block-convexity}
Let $F$ be a connected graph. Then the following statements hold.
\begin{itemize}
    \item[(1)] Every path in $F$ whose two endpoints lie in a block $B$ of $F$ is
    contained in $B$.
    \item[(2)] If $R$ is a connected subtree of $T(F)$ with
    $V(R)\cap\mathcal B(F)\neq\emptyset$, then the union of the blocks
    in $V(R)\cap\mathcal B(F)$ is a connected induced subgraph of $F$.
\end{itemize}
\end{lemma}

\begin{proof}
If a path with both endpoints in $B$ is not contained in $B$, then it contains a subpath $Q$ with distinct endpoints in $B$ and all
internal vertices outside $B$.
Hence $B\cup Q$ is a 2-connected subgraph of $F$ that strictly contains $B$. This contradiction proves~(1).

For (2), let $R$ be the subtree in the statement. Since $R$ is
connected, the definition of $T(F)$ implies that the union
in (2) is connected. To prove that it is induced, suppose that an edge $xy$ has both endpoints
in the union but the block containing $xy$, say $B_{xy}$, does not belong to $V(R)$. Let
$B_x,B_y\in V(R)\cap\mathcal B(F)$ contain $x,y$, respectively. If
$B_x=B_y$, then the two distinct blocks $B_x$ and $B_{xy}$ share the two
vertices $x,y$, a contradiction. If $B_x\neq B_y$, then
$B_x,x,B_{xy},y,B_y$ is a path in $T(F)$ different from the path in
$R$, contradicting the fact that such a path is unique in $T(F)$.
\end{proof}

The following lemma is used to construct the ``cut set'' $\mathcal{S}$ in Theorem~\ref{thm:del1cut}.

\begin{lemma}\label{lem:tree-separator}
Let $F$ be a connected graph and let $\ell\geq1$. If
$R_1,\ldots,R_\ell$ are pairwise
vertex-disjoint subtrees of the block-cut tree $T(F)$, each
containing at least one vertex of $\mathcal B(F)$, then there is a set
$Z\subseteq\mathcal C(F)$ with $|Z|\leq\ell-1$ such that every path in
$T(F)$ between two distinct subtrees $R_i,R_j$ meets $Z$.
\end{lemma}

\begin{proof}
In the block-cut tree $T(F)$, contract each subtree $R_i$ to a vertex
$r_i$, and root the resulting tree (say $T^*$) at $r_1$. For every $i\geq2$, let
$e_i$ be the edge of $T(F)$ corresponding to the first edge on the
path from $r_i$ to the root. Thus $e_i$ has exactly one endpoint in
$R_i$. Let $z_i$ be the endpoint of $e_i$ in $\mathcal C(F)$, and let
$Z:=\{z_i:2\leq i\leq\ell\}$.
Fix distinct $i,j\in[\ell]$. By symmetry, we may assume that
$r_i$ is not an ancestor of $r_j$. Then $i\geq2$, and every path
in $T(F)$ between $R_i$ and $R_j$ contains $e_i$ and hence
meets $Z$ at $z_i$. This proves the lemma.
\end{proof}

We are ready to present the proof of Theorem~\ref{thm:del1cut}.

\begin{proof}[\normalfont\textbf{Proof of Theorem~\ref{thm:del1cut}}]
We will complete the proof in two steps:
first, for every component $D$ of $G_X$ that contains a
path in $\mathcal Y$, we construct the corresponding ``local'' collections $\mathcal P_D,\mathcal H_D$, and a set $\mathcal S_D$;
second, we take their unions over all such components to obtain the collections $\mathcal P,\mathcal H$
and the set $\mathcal S$ in the theorem. 

\medskip

\noindent {\bf Step 1.} Fix a component $D$ of $G_X$ and let $\mathcal Y_D:=\{P\in\mathcal Y:P\subseteq D\}$. 
We will focus in particular on those blocks of $D$ that contain edges from at least two distinct paths in $\mathcal{Y}_D$. To this end, define
\[
 \mathcal B_D:=\{B\in\mathcal B(D):
     E(B)\cap E(P)\neq\emptyset\text{ for at least two paths }
     P\in\mathcal Y_D\}.
\]

For every nontrivial path $Q$ in $D$, let $T_Q$ be the minimal subtree
of the block-cut tree $T(D)$ containing every block $B$ for which
$E(B)\cap E(Q)\neq\emptyset$. These subtrees have the following properties:
\begin{itemize}
    \item[(P1)] $T_Q$ is a path whose endpoints belong to
    $\mathcal B(D)$;
    \item[(P2)]
    $V(T_Q)\cap\mathcal B(D)=\{B\in\mathcal B(D):
    E(B)\cap E(Q)\neq\emptyset\}$, and every cut-vertex of $D$ on
    $T_Q$ belongs to $V(Q)$;
    \item[(P3)] for every block $B$ of $D$, the set
    $E(Q)\cap E(B)$ is either empty or the edge set of a nontrivial
    subpath of $Q$.
\end{itemize}

To capture how the paths in $\mathcal{Y}_D$ are distributed among the blocks in $\mathcal{B}_D$, we introduce the following auxiliary graph.
Let $\mathcal F_D$ be the bipartite graph with bipartition
$(\mathcal Y_D,\mathcal B_D)$, where $P\in\mathcal Y_D$ is adjacent
to $B\in\mathcal B_D$ if and only if $E(P)\cap E(B)\neq\emptyset$. We claim
that $\mathcal F_D$ is a forest. Suppose for contradiction that it contains a
cycle $P_1B_1P_2B_2\cdots P_tB_tP_1$. For each $i$, let $R_i$ be the
nontrivial $(B_{i-1},B_i)$-subpath of $T_{P_i}$, where $B_0:=B_t$.
We show that the paths $R_i$ are pairwise edge-disjoint. Indeed, an edge common
to $R_i$ and $R_j$ would have an endpoint in $\mathcal C(D)$, and by (P2), this cut-vertex would belong to both $P_i$ and $P_j$, contrary to the fact that paths in $\mathcal Y_D$ are disjoint.
Therefore, the union of these $R_i$
contains a cycle in $T(D)$, a contradiction.
This proves that $\mathcal F_D$ is a forest.

Note that, by the definition of $\mathcal B_D$, every block in $\mathcal B_D$ has degree at least two in the
forest $\mathcal F_D$. So we have
\[
 2|\mathcal B_D|\leq e(\mathcal F_D)
 \leq |\mathcal B_D|+|\mathcal Y_D|-1.
\]
Hence $|\mathcal B_D|\leq|\mathcal Y_D|-1$ and
$e(\mathcal F_D)\leq2|\mathcal Y_D|-2$.

We next define the ``local'' path collection $\mathcal P_D$ by the union of the following two collections of paths. For every
$P\in\mathcal Y_D$ and $B\in\mathcal B_D$ with
$E(P)\cap E(B)\neq\emptyset$, (P3) shows
that $E(P)\cap E(B)$ is the edge set of a unique nontrivial subpath of
$P$. Let
\[
 \mathcal P_D^{(1)}:=
 \{Q:Q\text{ is a path and }E(Q)=E(P)\cap E(B)
 \text{ for some }P\in\mathcal Y_D\text{ and }B\in\mathcal B_D
 \text{ with }E(P)\cap E(B)\neq\emptyset\}.
\]
Next, let
\[
 \mathcal P_D^{(2)}:=
 \{Q:Q\text{ is a maximal nontrivial subpath of some }P\in\mathcal Y_D
 \text{ with }E(Q)\subseteq E(P)\setminus E(\mathcal P_D^{(1)})\},
\]
and set $\mathcal P_D:=\mathcal P_D^{(1)}\cup\mathcal P_D^{(2)}$.
The definitions imply that the paths in $\mathcal P_D$ are
edge-disjoint and that $E(\mathcal{Y}_D)$ is the disjoint union of
$E(\mathcal P_D^{(1)})$ and $E(\mathcal P_D^{(2)})$. Hence
$E(\mathcal P_D)=E(\mathcal Y_D)$.

For $P\in\mathcal Y_D$, write $\deg_{\mathcal F_D}(P)$ for its degree
in $\mathcal F_D$. The blocks $B\in\mathcal B_D$ for which
$E(P)\cap E(B)\neq\emptyset$ are exactly the neighbors of $P$ in
$\mathcal F_D$. Hence
$\mathcal P_D^{(1)}$ contains exactly $\deg_{\mathcal F_D}(P)$
subpaths of $P$, and therefore $\mathcal P_D^{(2)}$ contains at most
$\deg_{\mathcal F_D}(P)+1$ subpaths of $P$. 
Since $\mathcal F_D$ is bipartite with $\mathcal Y_D$ being one of the two parts, we have $|\mathcal{P}_D^{(1)}|=\sum_{P\in\mathcal Y_D}\deg_{\mathcal F_D}(P)=e(\mathcal F_D)$. 
Therefore, using the fact $e(\mathcal F_D)\leq2|\mathcal Y_D|-2$, we obtain
\begin{align}\label{eq: bound-Yd}
 |\mathcal P_D|= |\mathcal P^{(1)}_D|+|\mathcal P^{(2)}_D|
 \leq \sum_{P\in\mathcal Y_D}\deg_{\mathcal F_D}(P)
       +\sum_{P\in\mathcal Y_D}
          \bigl(\deg_{\mathcal F_D}(P)+1\bigr)=|\mathcal Y_D|+2e(\mathcal F_D)
 <5|\mathcal Y_D|.
 \end{align}

We next construct $\mathcal H_D$ and $\mathcal S_D$. Let $\mathcal R_D$ be the
following collection of subpaths of the block-cut tree $T(D)$:
\[
 \mathcal R_D:=
 \bigl\{\{B\}:B\in\mathcal B_D\bigr\}
 \cup\bigl\{T_Q:Q\in\mathcal P_D^{(2)}\bigr\}.
\]
We claim that the subpaths in $\mathcal R_D$ are pairwise
disjoint. By the definitions of $T_Q$ and $\mathcal P_D^{(2)}$, no such
$T_Q$ contains a block in $\mathcal B_D$.
Consider distinct $Q,Q'\in\mathcal P_D^{(2)}$. Since $Q$ and $Q'$ are vertex-disjoint, (P2) implies that $T_Q$ and $T_{Q'}$ share no cut-vertex of $D$. If they share a block $B$, then $B\notin\mathcal B_D$, so $Q$ and $Q'$ lie on the same path $P\in\mathcal Y_D$. By (P3) and the definition of $\mathcal P_D^{(2)}$, all edges of $P$ in $B$ belong to a single member of $\mathcal P_D^{(2)}$, forcing $Q=Q'$, a contradiction. Hence the members of $\mathcal R_D$ are pairwise disjoint.  
This proves the claim.

Define $\mathcal H_D$ to be the following collection of unions of blocks:
\[
 \mathcal H_D:=
\left\{\bigcup_{B\in V(R)\cap\mathcal B(D)}B:
 R\in\mathcal R_D\right\}.
\]
Since the paths in $\mathcal R_D$ are pairwise disjoint, the subgraphs in $\mathcal H_D$
are pairwise edge-disjoint. By (P1), every subpath in $\mathcal R_D$ contains a block. This gives a one-to-one correspondence between $\mathcal R_D$ and
$\mathcal H_D$. By Lemma~\ref{lem:block-convexity}(2), every subgraph in $\mathcal H_D$ is connected and induced.
Apply Lemma~\ref{lem:tree-separator} to $\mathcal R_D$, and let $\mathcal S_D$ be the resulting set of cut-vertices of $D$.
Then we have
\begin{align}\label{eq:bound-Hd}
    |\mathcal S_D|<|\mathcal H_D|=|\mathcal R_D|=|\mathcal{B}_D|+|\mathcal{P}_D^{(2)}|\leq |\mathcal B_D|+\sum_{P\in\mathcal Y_D}
          \bigl(\deg_{\mathcal F_D}(P)+1\bigr)=|\mathcal B_D|+|\mathcal Y_D|+e(\mathcal F_D)
       \leq4|\mathcal Y_D|.
\end{align}

\noindent {\bf Step 2.}
For the remainder of the proof, each union over $D$ always means that it ranges over the
components $D$ of $G_X$ that contain a path in $\mathcal Y$. Let
\[
 \mathcal P:=\bigcup_D\mathcal P_D,\qquad
 \mathcal H:=\bigcup_D\mathcal H_D,\qquad \mbox{and} \qquad
 \mathcal S:=\bigcup_D\mathcal S_D.
\]

We now verify each of the five conclusions in Theorem~\ref{thm:del1cut}. 
First, using $\sum_D|\mathcal Y_D|=|\mathcal Y|\leq m$, item~(1) is directly obtained by summing the bounds in~\eqref{eq: bound-Yd} and~\eqref{eq:bound-Hd}.

Consider~(2). If $Q\in\mathcal P_D^{(1)}$, then $E(Q)=E(P)\cap E(B)$ for some $P\in\mathcal Y_D$ and $B\in\mathcal B_D$,
and hence $Q\subseteq B$ and $B$ is an element in
$\mathcal H_D$. If $Q\in\mathcal P_D^{(2)}$, then $Q$ lies in the
subgraph
$\bigcup_{B\in V(T_Q)\cap\mathcal B(D)}B\in\mathcal H_D$.
Since the subgraphs in $\mathcal H_D$ are pairwise edge-disjoint in $D$, every
path in $\mathcal P_D$ lies in exactly one of them. The components $D$
are also pairwise disjoint. Hence every path in $\mathcal P$ lies in
exactly one subgraph in $\mathcal H$. Moreover,
$E(\mathcal P_D)=E(\mathcal Y_D)$ for every $D$, and therefore
$E(\mathcal P)=E(Y-X)$. This proves (2).

We next verify (3). By construction, each $H\in\mathcal H_D$
is a union of blocks in $D$ and contains at least one path
in $\mathcal P_D$. Moreover, distinct paths in $\mathcal P_D$
contained in the same $H$ are subpaths of distinct paths in
$\mathcal Y_D$. Since the paths in $\mathcal Y_D$ are pairwise
vertex-disjoint, so are these subpaths.
The components $D$ are pairwise disjoint, so every $H\in\mathcal H$ contains at least one path in $\mathcal P$, and the paths in
$\mathcal P$ contained in $H$ are pairwise vertex-disjoint.

By definition, $\mathcal S_D$ is a set of cut-vertices that separates every pair of distinct paths in $\mathcal{R}_D$ in the block-cut tree $T(D)$.
By the definition of the block-cut tree, $\mathcal S_D$ separates every pair of distinct subgraphs in $\mathcal{H}_D$.
Hence, $\mathcal{S}$ separates every pair of distinct subgraphs in $\mathcal{H}$. This verifies~(3).

We then consider (4). By definition of $\mathcal H_D$, every $H\in\mathcal H$ is a union of blocks. Moreover, if a subgraph in $\mathcal H_D$ contains at
least two paths in $\mathcal P_D$, then it is a block in
$\mathcal B_D$. 
Such a subgraph cannot be an edge or a single vertex, since it contains two vertex-disjoint nontrivial paths by (3).
This proves (4).

It remains to prove (5). 
Fix a cycle $C\subseteq H$ with $|C|=c(H)$.
Let $I_1,I_2$ be the two disjoint paths in the
hypothesis of (5), where $x_i\in V(X)$ is the endpoint of $I_i$ on $X$. We claim
that $H\cup I_1\cup I_2$ contains two disjoint
$(\{x_1,x_2\},C)$-paths. Otherwise, Menger's theorem gives a cut-vertex $w$ that separates $\{x_1,x_2\}$ from $C$ in $H\cup I_1\cup I_2$. Relabeling the two
paths, we may assume $w\notin V(I_1)$. 
Let $a$ be the endpoint of $I_1$ in $H$.
By~(4), $H$ is 2-connected, so there is an $(a,C)$-path $J$ in $H-w$.
Then $I_1\cup J$ contains an $(x_1,C)$-path that avoids $w$, a contradiction. This proves the claim.
Therefore, there are two disjoint $(X,C)$-paths $L_1$ and $L_2$ in $H\cup I_1\cup I_2$.
By relabeling, we assume that $x_i$ is an endpoint of $L_i$ for $i\in [2]$, and let $y_i$ be the endpoint of $L_i$ in $C$. 
Then the concatenation of $L_1$, $L_2$, and the longer one of the two $(y_1,y_2)$-paths on $C$ gives an $(x_1,x_2)$-path whose interior vertices all lie outside $X$.
By the maximality of $|X|$, the length of this path is at most $\operatorname{dist}_X(x_1,x_2)$. Hence, 
$\operatorname{dist}_X(x_1,x_2)\geq |C|/2$.
This completes the proof of Theorem~\ref{thm:del1cut}.
\end{proof}

\section{Theorem~\ref{thm:smith-linear} for large circumference: an amplifier covering}\label{sec:4}

Throughout this section, let $X$ and $Y$ be longest cycles in a $k$-connected graph $G$ with $m:=|V(X)\cap V(Y)|\geq2$. 
Let $G_X:=G-X$ and let $\mathcal Y$ be the
collection of nontrivial maximal paths of $Y-X$. Thus $|\mathcal Y|\leq m$. Let
$\mathcal P,\mathcal H,\mathcal S$ be the collections constructed in Theorem~\ref{thm:del1cut}. 
Recall the definition of an amplifier from Definition~\ref{def:amplifier}.

In this section, we first establish Lemma~\ref{lem:path-cover}, which gives a covering for all but $O(m)$ edges of the path
collection $\mathcal P$ using $O(m)$ amplifiers and some leftovers.
We further apply this covering result in Subsection~\ref{sec:main proof 1} to prove Theorem~\ref{thm:smith-linear} under the additional assumption that $c(G)=\Omega(mk)$.

\subsection{A covering result using amplifiers}\label{sec:path refinement}
\begin{lemma}\label{lem:path-cover}
    There exist collections of paths $\mathcal{U}$ and $\mathcal{V}$ in
    $G_X$ and a set $\mathcal S'\subseteq V(G_X)$ satisfying:
    \begin{itemize}
        \item[(1)] Every path in $\mathcal U\cup\mathcal V$ is a subpath
        of some $P\in\mathcal P$, and $
        \bigl|E(\mathcal P)\setminus
        \bigl(E(\mathcal U)\cup E(\mathcal V)\bigr)\bigr|\leq20m;$
        \item[(2)] $\mathcal S'$ contains $\mathcal S$ and the vertex
        of every trivial component of $Y-X$, and
        $|\mathcal S'|\leq115m, 
        |\mathcal U|\leq10m,
        |\mathcal V|\leq15m;$
        \item[(3)] Every $Q\in\mathcal U$ is contained in some
        $H\in\mathcal H$ with $|\mathcal P_H|\geq2$ and is an
        $(H,1/13)$-amplifier;
        \item[(4)] For every $Q\in\mathcal V$, the set
        $\mathcal S'$ separates $V(Q)$ from
        $V(\mathcal P)\setminus V(Q)$ in $G_X$.
    \end{itemize}
\end{lemma}

Recall that $\mathcal P_H$ is the collection of paths in $\mathcal P$
contained in $H$. For $H\in\mathcal H$ with $|\mathcal P_H|\geq2$
and $P\in\mathcal P_H$, let $\mathcal D_P$ be the collection of
components of $H-P$ that contain a path in
$\mathcal P_H\setminus\{P\}$.
The proof of Lemma~\ref{lem:path-cover} applies the following construction independently to certain paths $P\in\mathcal P$.

\begin{lemma}\label{lem:P-split}
    Let $H\in\mathcal H$ satisfy $|\mathcal P_H|\geq2$, and let
    $P\in\mathcal P_H$ satisfy $|P|\geq5$.
    Then there exist two collections $\mathcal U_P$ and $\mathcal V_P$
    of subpaths of $P$, and a vertex set $S_P\subseteq V(H)$, such that:
    \begin{itemize}
        \item[(1)] The paths in $\mathcal U_P\cup\mathcal V_P$ cover
        $E(P)$.
        Moreover, $|\mathcal U_P|\leq|\mathcal D_P|,
            |\mathcal V_P|\leq|\mathcal D_P|+1, \mbox{and } |S_P|\leq6|\mathcal D_P|+10.$
        \item[(2)] Every path in $\mathcal U_P$ is an
        $(H,1/13)$-amplifier;
        \item[(3)] For every $Q\in\mathcal V_P$,
        the set $S_P$ separates $V(Q)$ from
        $V(P-Q)$ and from
        $V(\mathcal P_H\setminus\{P\})$ in $H$.
    \end{itemize}
\end{lemma}

\begin{proof}
Let $a$ and $b$ be the endpoints of $P$ with $a <_P b$.
Let $a^+$ be the successor of $a$ and $b^-$ the predecessor of $b$ in
this order. Choose $c,d\in V(P)$ such that
$|P[a,c]|=|P[d,b]|=2$. Since $|P|\geq5$, we have $c<_P d$.

Let $\mathcal D_P=\{D_1,\ldots,D_\ell\}$, where
$\ell=|\mathcal D_P|$. By Theorem~\ref{thm:del1cut}(4), $H$ is 2-connected, so $D_i$ has at least two neighbors on $P$.
Let $u_i$ and $v_i$ be the first and last such neighbors in the order $<_P$, respectively.
Then $u_i<_P v_i$.
Since an $(u_i,v_i)$-path with internal vertices in $D_i$ together with
$P[u_i,v_i]$ forms a cycle, we see that $P[u_i,v_i]$ is an
$(H,1)$-amplifier. 

We now define two vertices $\overline{u_i}$ and $\overline{v_i}$. See Figure~\ref{fig:P-split}(I) for an illustration of $H$, where the shaded region represents $D_i$,
and the solid curves represent paths in $H$. The blue dashed curves indicate the association of $D_i$ with $\overline{u_i}$ and $\overline{v_i}$; they do not represent paths in $H$. 
If $u_i <_P c$ or there are no two disjoint $(P[a,u_i), P(u_i,b])$-paths in $H-u_i$, let $\overline{u_i} := u_i$.
Otherwise, let $\overline{u_i}$ be the minimum vertex $w \in V(P[c,u_i])$ (with respect to $<_P$) such that
for every $v\in V(P[w,u_i])$, the graph $H-v$ contains two disjoint
$(P[a,v),P(v,b])$-paths.
We define the vertex $\overline{v_i}$ symmetrically:
If $v_i >_P d$ or there are no two disjoint $(P[a,v_i), P(v_i,b])$-paths in $H-v_i$, let $\overline{v_i} := v_i$.
Otherwise, let $\overline{v_i}$ be the maximum vertex $w \in V(P[v_i,d])$ such that
for every $v\in V(P[v_i,w])$, the graph $H-v$ contains two disjoint
$(P[a,v),P(v,b])$-paths.

\begin{figure}[htb]
\centering
\scalebox{0.9}{%
\begin{minipage}{\textwidth}
\centering
\tikzset{
  every picture/.style={line width=0.75pt},
  relation/.style={
    draw=blue!70!black,
    line width=0.75pt,
    dash pattern={on 3pt off 2pt}
  }
}

\begin{tikzpicture}[x=0.75pt,y=0.75pt,yscale=-1,xscale=1]
\path[use as bounding box] (45,20) rectangle (630,210);
\node[font=\bfseries] at (60,35) {(I)};

\draw (79.25,156.75) -- (599.23,156.75);

\draw[fill=black!5,dash pattern={on 0.84pt off 2.51pt}]
  (263.88,49.62) .. controls (263.88,42.65) and (269.53,37) .. (276.49,37)
  -- (359.97,37) .. controls (366.94,37) and (372.59,42.65) .. (372.59,49.62)
  -- (372.59,87.47) .. controls (372.59,94.44) and (366.94,100.09) .. (359.97,100.09)
  -- (276.49,100.09) .. controls (269.53,100.09) and (263.88,94.44) .. (263.88,87.47)
  -- cycle;
\draw[draw=black!28,line width=0.45pt]
  (299.07,65.68) -- (309.50,50.00) -- (340.50,50.00) -- (347.71,69.01);
\draw[draw=black!28,line width=0.45pt]
  (309.50,50.00) -- (320.63,70.12) -- (340.50,50.00);

\draw[relation]
  (266.20,84.50) .. controls (225,92) and (171,128) .. (138.21,156.75);
\draw[relation]
  (370.30,84.50) .. controls (408,92) and (447,129) .. (471.26,156.75);

\draw (238.26,156.75) .. controls (251.53,114.52) and (246,102.31) .. (290.23,69.01);
\draw (394.70,156.75) .. controls (384.20,121.18) and (379.22,103.42) .. (352.14,70.12);
\draw (322.29,156.75) .. controls (333.89,122.29) and (315.65,113.41) .. (320.63,70.12);
\draw (276.41,72.89) -- (299.07,65.68);
\draw (310.68,70.12) -- (328.37,70.12);
\draw (347.71,69.01) -- (360.98,71.78);

\draw[dash pattern={on 0.84pt off 2.51pt}] (237.55,104.18) -- (238.26,205.54);
\draw[dash pattern={on 0.84pt off 2.51pt}] (396.20,98.08) -- (395.81,204.99);
\draw[dash pattern={on 4.5pt off 4.5pt}] (244.69,181.12) -- (389.94,181.12);
\draw[shift={(391.94,181.12)},rotate=180] (10.93,-3.29)
  .. controls (6.95,-1.4) and (3.31,-0.3) .. (0,0)
  .. controls (3.31,0.3) and (6.95,1.4) .. (10.93,3.29);
\draw[shift={(242.69,181.12)},rotate=0] (10.93,-3.29)
  .. controls (6.95,-1.4) and (3.31,-0.3) .. (0,0)
  .. controls (3.31,0.3) and (6.95,1.4) .. (10.93,3.29);

\draw[dash pattern={on 0.84pt off 2.51pt}] (138.21,117.29) -- (138.77,187.23);
\draw[dash pattern={on 0.84pt off 2.51pt}] (471.26,118.40) -- (471.81,188.34);

\draw[dash pattern={on 4.5pt off 4.5pt}] (141.87,181.12) -- (232.79,181.12);
\draw[shift={(234.79,181.12)},rotate=180] (10.93,-3.29)
  .. controls (6.95,-1.4) and (3.31,-0.3) .. (0,0)
  .. controls (3.31,0.3) and (6.95,1.4) .. (10.93,3.29);
\draw[shift={(139.87,181.12)},rotate=0] (10.93,-3.29)
  .. controls (6.95,-1.4) and (3.31,-0.3) .. (0,0)
  .. controls (3.31,0.3) and (6.95,1.4) .. (10.93,3.29);

\draw[dash pattern={on 4.5pt off 4.5pt}] (401.67,181.12) -- (466.06,181.12);
\draw[shift={(468.06,181.12)},rotate=180] (10.93,-3.29)
  .. controls (6.95,-1.4) and (3.31,-0.3) .. (0,0)
  .. controls (3.31,0.3) and (6.95,1.4) .. (10.93,3.29);
\draw[shift={(399.68,181.12)},rotate=0] (10.93,-3.29)
  .. controls (6.95,-1.4) and (3.31,-0.3) .. (0,0)
  .. controls (3.31,0.3) and (6.95,1.4) .. (10.93,3.29);

\draw (158.48,156.75) .. controls (174.94,138.54) and (195.57,142.24) .. (207.36,156.75);
\draw (185.99,156.75) .. controls (200.73,176.28) and (233.16,175.17) .. (248.64,156.75);
\draw (131.58,156.75) .. controls (152.46,175.91) and (181.57,176.28) .. (199.99,156.75);
\draw (117.12,156.75) .. controls (131.86,175.77) and (165.72,175.54) .. (181.20,156.75);
\draw (192.71,156.75) .. controls (213.58,176.03) and (239.43,176.65) .. (257.85,156.75);
\draw (171.75,156.75) .. controls (188.20,138.54) and (208.84,142.24) .. (220.63,156.75);
\draw (436.21,156.75) .. controls (450.95,175.91) and (463.78,172.95) .. (478.16,156.75);
\draw (381.80,156.75) .. controls (402.68,175.54) and (431.79,175.91) .. (450.22,156.75);
\draw (368.78,156.75) .. controls (383.52,176.28) and (415.95,175.17) .. (431.42,156.75);
\draw (442.93,156.75) .. controls (463.81,175.66) and (489.65,176.28) .. (508.07,156.75);
\draw (406.12,156.75) .. controls (422.58,138.91) and (443.22,142.61) .. (455.01,156.75);
\draw (419.39,156.75) .. controls (435.85,138.91) and (456.48,142.61) .. (468.27,156.75);

\foreach \x in {79.25,86.875,94.50,138.21,238.26,322.29,394.70,471.26,583.77,591.50,599.23}
  \fill (\x,156.75) circle (1.8pt);

\node[anchor=north west,inner sep=0.75pt] at (67.67,125.57) {$P$};
\node[anchor=north west,inner sep=0.75pt] at (231.53,40.16) {$D_i$};
\node[anchor=north west,inner sep=0.75pt,font=\scriptsize] at (224.76,144.67) {$u_i$};
\node[anchor=north west,inner sep=0.75pt,font=\scriptsize] at (397.78,144.67) {$v_i$};
\node[anchor=north west,fill=white,inner sep=1.2pt,font=\scriptsize]
  at (282.55,187.19) {$(H,1)$-amplifier};
\node[anchor=north west,fill=white,inner sep=0.5pt,font=\scriptsize]
  at (120.45,144.67) {$\overline{u_i}$};
\node[anchor=north west,fill=white,inner sep=1.2pt,font=\scriptsize]
  at (150.00,188.50) {$(H,1/6)$-amplifier};
\node[anchor=north west,fill=white,inner sep=1.2pt,font=\scriptsize]
  at (402.00,188.50) {$(H,1/6)$-amplifier};
\node[anchor=north west,fill=white,inner sep=0.5pt,font=\scriptsize]
  at (473.62,145.22) {$\overline{v_i}$};
\node[anchor=base,inner sep=0pt,font=\scriptsize] at (71.50,169.00) {$a$};
\node[anchor=base,inner sep=0pt,font=\scriptsize] at (86.875,169.00) {$a^+$};
\node[anchor=base,inner sep=0pt,font=\scriptsize] at (102.25,169.00) {$c$};
\node[anchor=base,inner sep=0pt,font=\scriptsize] at (575.50,169.00) {$d$};
\node[anchor=base,inner sep=0pt,font=\scriptsize] at (591.50,169.00) {$b^-$};
\node[anchor=base,inner sep=0pt,font=\scriptsize] at (607.50,169.00) {$b$};
\end{tikzpicture}

\par\vspace{1.2em}

\begin{tikzpicture}[x=0.75pt,y=0.75pt,yscale=-1,xscale=1]
\path[use as bounding box] (45,20) rectangle (630,210);
\node[font=\bfseries] at (60,35) {(II)};

\draw[red] (79.25,119) -- (125.33,119);
\draw       (125.33,119) -- (285.83,119);
\draw[red] (285.83,119) -- (331.33,119);
\draw       (331.33,119) -- (424.33,119);
\draw[red] (424.33,119) -- (461.83,119);
\draw       (461.83,119) -- (551.14,119);
\draw[red] (551.14,119) -- (599.23,119);

\draw[fill=black!5,dash pattern={on 0.84pt off 2.51pt}]
  (149,38.10) .. controls (149,32.80) and (153.30,28.50) .. (158.60,28.50)
  -- (216.23,28.50) .. controls (221.54,28.50) and (225.83,32.80) .. (225.83,38.10)
  -- (225.83,66.90) .. controls (225.83,72.20) and (221.54,76.50) .. (216.23,76.50)
  -- (158.60,76.50) .. controls (153.30,76.50) and (149,72.20) .. (149,66.90)
  -- cycle;
\draw[fill=black!5,dash pattern={on 0.84pt off 2.51pt}]
  (210,161.10) .. controls (210,155.80) and (214.30,151.50) .. (219.60,151.50)
  -- (277.23,151.50) .. controls (282.54,151.50) and (286.83,155.80) .. (286.83,161.10)
  -- (286.83,189.90) .. controls (286.83,195.20) and (282.54,199.50) .. (277.23,199.50)
  -- (219.60,199.50) .. controls (214.30,199.50) and (210,195.20) .. (210,189.90)
  -- cycle;
\draw[fill=black!5,dash pattern={on 0.84pt off 2.51pt}]
  (325,36.10) .. controls (325,30.80) and (329.30,26.50) .. (334.60,26.50)
  -- (392.23,26.50) .. controls (397.54,26.50) and (401.83,30.80) .. (401.83,36.10)
  -- (401.83,64.90) .. controls (401.83,70.20) and (397.54,74.50) .. (392.23,74.50)
  -- (334.60,74.50) .. controls (329.30,74.50) and (325,70.20) .. (325,64.90)
  -- cycle;
\draw[fill=black!5,dash pattern={on 0.84pt off 2.51pt}]
  (478.50,162.60) .. controls (478.50,157.30) and (482.80,153) .. (488.10,153)
  -- (535.04,153) .. controls (540.34,153) and (544.64,157.30) .. (544.64,162.60)
  -- (544.64,191.40) .. controls (544.64,196.70) and (540.34,201) .. (535.04,201)
  -- (488.10,201) .. controls (482.80,201) and (478.50,196.70) .. (478.50,191.40)
  -- cycle;
\draw[fill=black!5,dash pattern={on 0.84pt off 2.51pt}]
  (550.90,37.43) .. controls (550.90,32.13) and (555.19,27.83) .. (560.50,27.83)
  -- (614.96,27.83) .. controls (620.26,27.83) and (624.56,32.13) .. (624.56,37.43)
  -- (624.56,66.23) .. controls (624.56,71.54) and (620.26,75.83) .. (614.96,75.83)
  -- (560.50,75.83) .. controls (555.19,75.83) and (550.90,71.54) .. (550.90,66.23)
  -- cycle;

\draw[relation]
  (125.33,119) .. controls (129.00,102.50) and (144.50,82.50) .. (155.80,74.20);
\draw[relation]
  (221.33,119) .. controls (221.00,102.50) and (220.30,84.00) .. (219.03,74.20);
\draw[relation]
  (205.83,119) .. controls (205.90,132.50) and (210.30,146.00) .. (217.00,153.80);
\draw[relation]
  (285.83,119) .. controls (285.70,132.50) and (283.50,146.00) .. (279.83,153.80);
\draw[relation]
  (331.33,119) .. controls (331.45,102.50) and (331.60,83.00) .. (331.80,72.20);
\draw[relation]
  (424.33,119) .. controls (420.00,102.50) and (406.50,82.50) .. (395.03,72.20);
\draw[relation]
  (461.83,119) .. controls (462.00,133.00) and (473.00,147.00) .. (485.20,155.30);
\draw[relation]
  (551.14,119) .. controls (551.00,133.00) and (545.00,147.00) .. (537.94,155.30);
\draw[relation]
  (591.50,119) .. controls (580.00,103.50) and (563.50,84.50) .. (557.60,73.50);
\draw[relation]
  (599.23,119) .. controls (603.00,103.50) and (613.50,84.50) .. (617.86,73.50);

\foreach \x in {79.25,86.875,94.50,125.33,205.83,221.33,285.83,331.33,424.33,461.83,551.14,583.77,591.50,599.23}
  \fill (\x,119) circle (1.8pt);

\node[anchor=north west,inner sep=0.75pt] at (67.67,75.00) {$P$};
\node[anchor=north west,inner sep=0.75pt] at (125.50,35.50) {$D_1$};
\node[anchor=north west,inner sep=0.75pt] at (189.00,166.50) {$D_2$};
\node[anchor=north west,inner sep=0.75pt] at (303.50,34.50) {$D_3$};
\node[anchor=north west,inner sep=0.75pt] at (458.00,162.00) {$D_4$};
\node[anchor=north west,inner sep=0.75pt] at (527.08,40.14) {$D_5$};

\node[anchor=south east,inner sep=0.2pt,font=\tiny]
  at (122.80,116.20) {$\overline{u_1}$};
\node[anchor=south west,inner sep=0.2pt,font=\tiny]
  at (223.80,116.20) {$\overline{v_1}$};
\node[anchor=south east,inner sep=0.2pt,font=\tiny]
  at (203.30,116.20) {$\overline{u_2}$};
\node[anchor=south west,inner sep=0.2pt,font=\tiny]
  at (288.30,116.20) {$\overline{v_2}$};
\node[anchor=south east,inner sep=0.2pt,font=\tiny]
  at (328.80,116.20) {$\overline{u_3}$};
\node[anchor=south west,inner sep=0.2pt,font=\tiny]
  at (426.80,116.20) {$\overline{v_3}$};
\node[anchor=south east,inner sep=0.2pt,font=\tiny]
  at (459.30,116.20) {$\overline{u_4}$};
\node[anchor=south west,inner sep=0.2pt,font=\tiny]
  at (553.60,116.20) {$\overline{v_4}$};

\node[anchor=north,inner sep=0.4pt,font=\scriptsize] at (102.29,124.00) {$Q_1$};
\node[anchor=north,inner sep=0.4pt,font=\scriptsize] at (308.58,124.00) {$Q_2$};
\node[anchor=north,inner sep=0.4pt,font=\scriptsize] at (443.08,124.00) {$Q_3$};
\node[anchor=north,inner sep=0.4pt,font=\scriptsize] at (575.19,124.00) {$Q_4$};

\node[anchor=base,inner sep=0pt,font=\scriptsize] at (71.50,109.00) {$a$};
\node[anchor=base,inner sep=0pt,font=\scriptsize] at (86.875,109.00) {$a^+$};
\node[anchor=base,inner sep=0pt,font=\scriptsize] at (102.25,109.00) {$c$};
\node[anchor=base,inner sep=0pt,font=\scriptsize] at (571.00,109.00) {$d$};
\node[anchor=base,inner sep=0pt,font=\scriptsize] at (591.50,109.00) {$b^-$};
\node[anchor=base,inner sep=0pt,font=\scriptsize] at (612.00,109.00) {$b$};
\end{tikzpicture}

\end{minipage}
}
\caption{Proof of Lemma~\ref{lem:P-split}.}
\label{fig:P-split}
\end{figure}
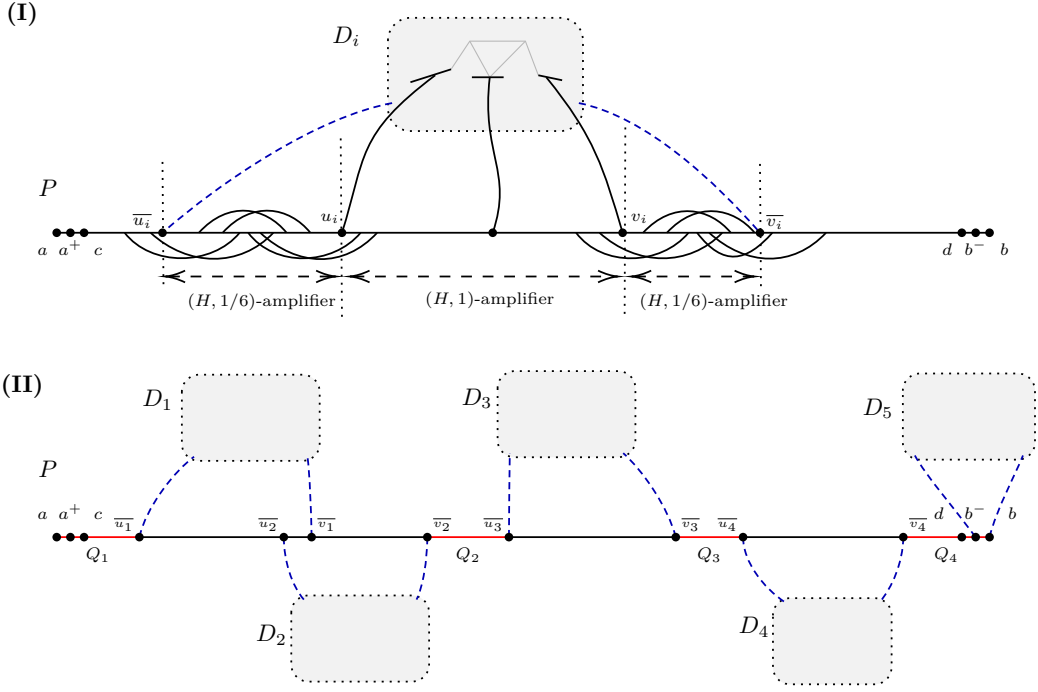

We claim that if $u_i \leq_P d$ and $v_i \geq_P c$, then
$P[\overline{u_i},\overline{v_i}]$ is an $(H,1/13)$-amplifier.
If $\overline{u_i}<_P u_i$, then
$c\leq_P\overline{u_i}<_P u_i\leq_P d$, and thus
$|P[a,\overline{u_i}]|,|P[u_i,b]|\geq 2$. By Lemma~\ref{lem:local-amplifier}, $P[\overline{u_i},u_i]$ is an $(H,1/6)$-amplifier, since the choice of $\overline{u_i}$ ensures that, for every $v\in V(P[\overline{u_i},u_i])$, the graph $H-v$ contains two disjoint $(P[a,v),P(v,b])$-paths.
The symmetric argument shows that
$P[v_i,\overline{v_i}]$ is also an $(H,1/6)$-amplifier whenever $v_i<_P\overline{v_i}$.
Recall that $P[u_i,v_i]$ is an $(H,1)$-amplifier.
Write $\ell_1,\ell_2,\ell_3$ for the lengths of the three consecutive
subpaths $P[\overline{u_i},u_i]$, $P[u_i,v_i]$, and
$P[v_i,\overline{v_i}]$. Then there are cycles in $H$ containing at least
$\ell_1/6$, $\ell_2$, and $\ell_3/6$ edges of these 
subpaths, respectively. Since $\max\{\ell_1/6,\ell_2,\ell_3/6\}
\geq(\ell_1+\ell_2+\ell_3)/13$, one of the three cycles
contains at least $1/13$ of the edges of
$P[\overline{u_i},\overline{v_i}]$, proving the claim.

Let $\mathcal U_P$ be the set of distinct paths
$P[\overline{u_i},\overline{v_i}]$ provided by the claim, as $i$ ranges
over the indices in $[\ell]$ satisfying $u_i\leq_P d$ and
$v_i\geq_P c$. Then $|\mathcal U_P|\leq\ell$.
Delete all edges in $E(\mathcal U_P)$ from $P$, and let
$\mathcal V_P$ be the collection of nontrivial paths that
remain. As shown in Figure~\ref{fig:P-split}(II), all paths in $\mathcal V_P$
are drawn in red. Since $|\mathcal U_P|\leq\ell$, we have
$|\mathcal V_P|\leq\ell+1$, and the paths in
$\mathcal U_P\cup\mathcal V_P$ cover $E(P)$.
Moreover, for each $Q=P[x,y]\in\mathcal V_P$, where $x<_P y$, we have
\[
x \in \{\overline{v_i} : \overline{v_i} \geq_P c\} \cup \{a\},
\qquad
y \in \{\overline{u_i} : \overline{u_i} \leq_P d\} \cup \{b\}.
\]

We claim that if $x$ is the left endpoint of some $Q\in\mathcal V_P$,
then there is a set $Z_x$ of at most three vertices that contains $x$
and separates $P[a,x]$ from $P[x,b]$ in $H$.
If $x=a$, then $\{a\}$ is such a set. Otherwise,
$x=\overline{v_i}$ for some $i$. If $x\geq_Pd$, then
$V(P[d,b])$ has the required property. It remains to consider
$x<_Pd$.
If $H-v_i$ has no two disjoint
$(P[a,v_i),P(v_i,b])$-paths, then $x=v_i$. By Menger's theorem in
$H-x$, there is a cut-vertex separating these two
sets in $H-x$. Hence the union of this cut-vertex and $x$ forms a set of size at most two separating
$P[a,x]$ from $P[x,b]$ in $H$.
In the remaining case, let $x^+$ be the successor of $x$ on $P$. Since
$x^+\in V(P[v_i,d])$, the maximality of
$x=\overline{v_i}$ implies that $H-x^+$ has no two disjoint
$(P[a,x^+),P(x^+,b])$-paths. By Menger's theorem, there is a cut-vertex separating these sets in $H-x^+$. It follows that the union of this cut-vertex and 
$\{x,x^+\}$ forms a set of size at most three separating $P[a,x]$
from $P[x,b]$ in $H$.
This proves the claim.
By symmetry, for every right endpoint $y$ of a path in
$\mathcal V_P$, there is a set $Z_y$ of at most three vertices that
contains $y$ and separates $P[a,y]$ from $P[y,b]$ in $H$.

For each
$Q=P[x,y]\in\mathcal V_P$, choose the sets $Z_x$ and $Z_y$ as above, and
let $S_P$ be the union of $\{a,a^+,b^-,b\}$ and all these sets $Z_x$ and $Z_y$. 
Then $|S_P|\leq6|\mathcal V_P|+4\leq6\ell+10$. 
It suffices to verify~(3).
For every
$Q=P[x,y]\in\mathcal V_P$, since $S_P$ contains $Z_x$ and $Z_y$,  $S_P$ separates $V(Q)$ from
$V(P-Q)$ in $H$.
It remains to show that $S_P$ separates $Q$ from all other paths in $\mathcal{P}_H\setminus\{P\}$ in $H$.
Let $R\in \mathcal{P}_H\setminus\{P\}$ lie in $D_i$. Suppose for contradiction that some
$(Q,R)$-path avoids $S_P$.
Then this path contains a $(P,R)$-subpath in $H-S_P$.
Let $z$ be the endpoint on $P$ of this $(P,R)$-subpath.
Then $z$ has a neighbor in $D_i$, and hence $z\in  V(P[u_i,v_i])$.
The complementary part of the $(Q,R)$-path above is a $(Q,z)$-path avoiding $S_P$. Since $S_P$ separates $Q$ from $P-Q$, we have
$z\in V(Q)$. Moreover, $z$ is not an endpoint of $Q$, since both endpoints $x$ and $y$ belong to $S_P$.
When $u_i\le_Pd$ and $v_i\ge_Pc$, this is impossible since the interior of $Q$ is disjoint from $V(P[u_i,v_i])$.
In the remaining case, we must have $z\in \{a,a^+,b^-,b\}\subseteq S_P$, also a contradiction.
Thus, $S_P$ separates $Q$ from every other path in $\mathcal P_H$, completing the
proof of Lemma~\ref{lem:P-split}.
\end{proof}

\begin{proof}[\normalfont\textbf{Proof of Lemma~\ref{lem:path-cover}}]
We first prove that
\begin{equation}\label{eq:component-sum}
 \sum_{\substack{H\in\mathcal H\\|\mathcal P_H|\geq2}}
       \ \sum_{P\in\mathcal P_H}|\mathcal D_P|\leq10m.
\end{equation}
Fix $H\in\mathcal H$ with $r:=|\mathcal P_H|\geq2$. Let $\mathcal F$
be the auxiliary graph obtained from $H$ by contracting every path
$P\in\mathcal P_H$ to a vertex $x_P$, deleting all loops and parallel edges. Consider its block-cut tree $T(\mathcal F)$.

For each $P\in\mathcal P_H$, assign weight one to the vertex
$x_P$ of $T(\mathcal F)$ if $x_P$ is a cut-vertex of $\mathcal F$, and otherwise to the
unique block containing $x_P$, viewed as a vertex of $T(\mathcal F)$. Note that a cut-vertex must have weight one, while a block may receive weight more than one.
This gives a weight for each vertex in $T(\mathcal F)$, and the total weight is $r$.

Since contracting the paths in $\mathcal P_H\setminus\{P\}$ neither splits nor merges the components of $H-P$, the components in $\mathcal D_P$ are in bijection with the
components of $\mathcal F-x_P$ that contain some vertex $x_Q$ with $Q\in\mathcal P_H\setminus\{P\}$. 
Hence, if $x_P$ is a cut-vertex of
$\mathcal F$, then $|\mathcal{D}_P|$ equals the number of components of
$T(\mathcal F)-x_P$ that contain a vertex with positive weight; and if $x_P$ is not a cut-vertex, then
$|\mathcal D_P|=1$.
Let $T_0$ be the minimal subtree of $T(\mathcal F)$ containing all vertices with positive weight. It follows from the preceding correspondence
that whenever $x_P$ is a cut-vertex of $\mathcal F$, we have
\begin{align}\label{eq:Dp}
    |\mathcal D_P|=\deg_{T_0}(x_P).
\end{align}

If $T_0$ is a single vertex, then $T_0$ represents a block since the total weight is $r\geq 2$. Thus no $x_P$ is a cut-vertex, and $\sum_{P\in\mathcal P_H}|\mathcal D_P|=r\leq2r-2$.
Now suppose that $|V(T_0)|\geq2$. Let
$\mathcal C_0:=\{x_P:P\in\mathcal P_H,\ x_P\text{ is a cut-vertex of }\mathcal F\}$,
let $c:=|\mathcal C_0|$, and let $b$ be the number of blocks (each as a vertex) in $T_0$ with
positive weight.
Since each of these $b+c$ vertices has positive weight, we have $b+c\leq r$.
Moreover, every vertex of $T_0$
with weight zero has degree at least two in $T_0$, since the minimality of $T_0$ implies that such a vertex is contained in a subpath between two vertices with positive weight.
Hence,
\[
\begin{aligned}
 \sum_{\substack{P\in\mathcal P_H\\x_P\in \mathcal C_0}}
       \deg_{T_0}(x_P)
 &=2\bigl(|V(T_0)|-1\bigr)
   -\sum_{z\in V(T_0)\setminus\mathcal C_0}\deg_{T_0}(z)\\
 &\leq
   2\bigl(|V(T_0)|-1\bigr)
   -b-2\bigl(|V(T_0)|-c-b\bigr)=2c+b-2.
\end{aligned}
\]
Together with~\eqref{eq:Dp}, this and the fact that $x_P$ is not a cut-vertex for $r-c$ paths $P\in\mathcal{P}_H$ imply that
\[
\begin{aligned}
 \sum_{P\in\mathcal P_H}|\mathcal D_P|
 &=\sum_{\substack{P\in\mathcal P_H\\x_P\in \mathcal C_0}}
       \deg_{T_0}(x_P)+(r-c)\leq (2c+b-2)+(r-c)
 \leq 2r-2.
\end{aligned}
\]
Recall that $|\mathcal{P}|\leq 5m$ by Theorem~\ref{thm:del1cut}. Summing over all $H\in\mathcal H$ with $|\mathcal P_H|\geq2$ gives
\[
\begin{aligned}
 \sum_{\substack{H\in\mathcal H\\|\mathcal P_H|\geq2}}
       \ \sum_{P\in\mathcal P_H}|\mathcal D_P|\leq
 \sum_{\substack{H\in\mathcal H\\|\mathcal P_H|\geq2}}
       \bigl(2|\mathcal P_H|-2\bigr)\leq2\sum_{H\in\mathcal H}|\mathcal P_H|
 =2|\mathcal P|\leq10m,
\end{aligned}
\]
which proves~\eqref{eq:component-sum}.

Define $\mathcal S'$ to be the union of the following (possibly intersecting) vertex sets:
\begin{itemize}
    \item[(i)] the vertex set consisting of the trivial components of $Y-X$;
    \item[(ii)] the set $\mathcal S$ from
    Theorem~\ref{thm:del1cut}; and
    \item[(iii)] the sets $S_P$ from Lemma~\ref{lem:P-split}, for all
    $H\in\mathcal H$ with $|\mathcal P_H|\geq2$ and all
    $P\in\mathcal P_H$ with $|P|\geq5$.
\end{itemize}
The vertices in (i) contribute at most $m$ to $|\mathcal S'|$,
and by Theorem~\ref{thm:del1cut}, vertices in~(ii) contribute at most $4m$. For (iii), Lemma~\ref{lem:P-split} and
\eqref{eq:component-sum} give a total contribution of at most
$6\cdot10m+10|\mathcal P|\leq110m$.
Hence,
\begin{align}\label{eq:S'}
    |\mathcal S'|
 \leq m+4m+110m=115m.
\end{align}

We construct $\mathcal U$ and $\mathcal V$ as follows.
\begin{itemize}
    \item If $H\in\mathcal H$ satisfies $|\mathcal P_H|=1$, then add the
    unique path in $\mathcal P_H$ to $\mathcal V$.
    \item If $|\mathcal P_H|\geq2$ and
    $P\in\mathcal P_H$ has length at least five, then add the paths in
    $\mathcal U_P$ to $\mathcal U$ and those in $\mathcal V_P$ to
    $\mathcal V$.
\end{itemize}
Then only paths $P\in\mathcal P$ of length at most four may contain
edges not covered by paths in $\mathcal U\cup\mathcal V$. Therefore,
\[
 \bigl|E(\mathcal P)\setminus
 (E(\mathcal U)\cup E(\mathcal V))\bigr|
 \leq4|\mathcal P|\leq20m,
\]
which proves (1).

For every $H\in\mathcal H$ with $|\mathcal P_H|\geq2$ and every
$P\in\mathcal P_H$ with $|P|\geq5$, Lemma~\ref{lem:P-split} adds at
most $|\mathcal D_P|$ paths to $\mathcal U$ and at most
$|\mathcal D_P|+1$ paths to $\mathcal V$. Each $H\in\mathcal H$ with
$|\mathcal P_H|=1$ contributes one further path to $\mathcal V$.
Consequently, \eqref{eq:component-sum} gives
\[
 |\mathcal U|
 \leq\sum_{\substack{H\in\mathcal H\\|\mathcal P_H|\geq2}}
       \sum_{\substack{P\in\mathcal P_H\\|P|\geq5}}
       |\mathcal D_P|
 \leq10m
\]
and
\[
 |\mathcal V|
 \leq
 \sum_{\substack{H\in\mathcal H\\|\mathcal P_H|\geq2}}
 \sum_{\substack{P\in\mathcal P_H\\|P|\geq5}}
 \bigl(|\mathcal D_P|+1\bigr)
 +\sum_{\substack{H\in\mathcal H\\|\mathcal P_H|=1}}1
 \leq\sum_{\substack{H\in\mathcal H\\|\mathcal P_H|\geq2}}
       \sum_{\substack{P\in\mathcal P_H\\|P|\geq5}}
       |\mathcal D_P|+|\mathcal P|
 \leq10m+5m=15m.
\]
Together with~\eqref{eq:S'}, this proves (2); property
(3) follows from Lemma~\ref{lem:P-split}(2).

To prove (4), let $Q\in\mathcal V$ be contained in $H\in\mathcal H$.
Since $\mathcal S\subseteq\mathcal S'$, Theorem~\ref{thm:del1cut}(3) shows that
$\mathcal S'$ separates $Q$ from every path in
$\mathcal P\setminus\mathcal P_H$ in $G_X$. If
$|\mathcal P_H|=1$, this proves (4). If
$|\mathcal P_H|\geq2$, Lemma~\ref{lem:P-split}(3) shows that
$\mathcal S'$ separates $Q$ from
$V(\mathcal P_H)\setminus V(Q)$ in $H$. By Theorem~\ref{thm:del1cut}(4), $H$ is a block of $G_X$. It follows from Lemma~\ref{lem:block-convexity}(1) that $Q$ and
$V(\mathcal P_H)\setminus V(Q)$ are also separated in $G_X$.
This proves (4) and completes the
proof of Lemma~\ref{lem:path-cover}. 
\end{proof}

\subsection{The large-circumference case of Theorem~\ref{thm:smith-linear}}\label{sec:main proof 1}

In this subsection, we aim to prove Theorem~\ref{thm:smith-linear} assuming $c(G)>23mk$ (i.e., Lemma~\ref{lem:large-circ}).
In addition to the notation introduced at the beginning of this section, let the collections $\mathcal{U},\mathcal{V},\mathcal{S}'$ be as given by Lemma~\ref{lem:path-cover}.

The next lemma is derived from the maximality of $|X|$ and $|Y|$.

\begin{lemma}\label{lem:XY-path}
For every $P\in\mathcal Y$, there are at most $2m$ disjoint $(X,Y)$-paths whose endpoint on $Y$ lies in $V(P)$.
\end{lemma}

\begin{proof}
Suppose for contradiction that $2m+1$ such paths have their endpoints
on $Y$ in $V(P)$.
Then at least $m+1$ of them avoid all vertices in $V(X)\cap V(Y)$. 
By the pigeonhole principle, there are two of them, say $L_1$ and $L_2$, that connect the
same subpath $R$ of $X-Y$ to $P$. 
Write their endpoints as
$u_1,u_2\in V(R)$ and $v_1,v_2\in V(P)$. By symmetry between $X$ and
$Y$, we may assume that $|R[u_1,u_2]|\leq |P[v_1,v_2]|$. Replacing
$R[u_1,u_2]$ in $X$ by $L_1\cup P[v_1,v_2]\cup L_2$ gives a cycle
longer than $X$, a contradiction.
\end{proof}

We next bound the lengths of the paths in $\mathcal U$ and
$\mathcal V$.

\begin{lemma}\label{lem:cover-lengths}
Suppose that $k>117m$. Then
\begin{itemize}
    \item[(1)] Every $P\in\mathcal U$ satisfies
    $|P|\leq\max\{26c(G)/k,k\}$.
    \item[(2)] Every $Q\in\mathcal V$ satisfies $|Q|\leq117m$.
\end{itemize}
\end{lemma}

\begin{proof}
Suppose~(1) fails for some $P\in\mathcal U$, and let $H\in\mathcal H$ contain $P$.
By Lemma~\ref{lem:path-cover}(3), we have
\[
    c(H)\geq\frac{|P|}{13}
       >\frac{2c(G)}{k}.
\]
Moreover, $|V(P)|>|P|>k$. By Theorem~\ref{thm:dirac},
$|V(X)|\geq k$, so
Menger's theorem gives $k$ disjoint $(X,P)$-paths in $G$. Each of these $(X,P)$-paths contains an
$(X,H)$-subpath with the same endpoint on $X$. List their endpoints on $X$ as
$x_1,\ldots,x_k$ in cyclic order on $X$, where $x_{k+1}:=x_1$. For every $i\in[k]$,
Theorem~\ref{thm:del1cut}(5) gives
$\operatorname{dist}_X(x_i,x_{i+1})\geq c(H)/2>c(G)/k$.
Summing over $i$ gives $c(G)=|X|>c(G)$, a contradiction.

For (2), suppose for contradiction that some $Q\in\mathcal V$ has $|Q|>117m$. Since
$k>117m$, Menger's theorem gives $117m+1$ disjoint $(X,Q)$-paths. At
least $117m+1-|\mathcal S'|\geq2m+1$ of them avoid
$\mathcal S'$. 
By Lemma~\ref{lem:path-cover}(2), these paths avoid the vertex of every trivial paths in $\mathcal Y$, since they avoid $\mathcal S'$.
By Lemma~\ref{lem:path-cover}(4), they also avoid
$V(\mathcal P)\setminus V(Q)$. 
Therefore, these $2m+1$ paths are indeed $(X,Y)$-paths whose endpoints on $Y$ all lie in the path in $\mathcal{Y}$ that contains $Q$.
This contradicts Lemma~\ref{lem:XY-path} and proves (2).
\end{proof}

We now prove Theorem~\ref{thm:smith-linear} under the assumption $c(G)=\Omega(mk).$

\begin{lemma}\label{lem:large-circ}
If $c(G)>23mk$, then $m\geq k/600$.
\end{lemma}

\begin{proof}
Suppose for contradiction that $k>600m$. In particular,
$k>117m$, so Lemma~\ref{lem:cover-lengths} applies. By Theorem~\ref{thm:del1cut}(2), $|E(\mathcal P)|=|E(Y-X)|\geq c(G)-2m$. The bounds in
Lemmas~\ref{lem:path-cover} and~\ref{lem:cover-lengths} give
\begin{align*}
c(G)-2m
&\leq |E(\mathcal P)|\leq |E(\mathcal U)|+|E(\mathcal V)|+20m\\
&\leq10m\cdot \max\left\{\frac{26c(G)}{k},k\right\}
       +15m\cdot117m+20m\\
&\leq\frac{260m}{k}c(G)+10mk+1755m^2+20m.
\end{align*}
Since $m\geq2$ and $k>600m$, we have $1755m^2+22m<1800m^2<3mk$. Hence
\[
 c(G)<\frac{260m}{k}c(G)+13mk
      <\left(\frac{13}{30}+\frac{13}{23}\right)c(G)
      =\frac{689}{690}c(G),
\]
a contradiction. This proves Lemma~\ref{lem:large-circ}.
\end{proof}

\section{A higher-connectivity extension of Dirac's theorem}\label{sec:5}
In this section, we prove an extension of Dirac's theorem for highly connected graphs $G$ (i.e., Theorem~\ref{thm:small-circ}). 
This provides the final piece needed for the proof of Theorem~\ref{thm:smith-linear}, which we complete in the next section.

The celebrated theorem of Dirac~\cite{dirac1952some}
states that every $2$-connected graph $G$ satisfies
$c(G)\geq\min\{2\delta(G),|V(G)|\}$.
Motivated by this theorem,
Jung~\cite{jung2001degree} conjectured the following sharp generalization for highly connected graphs: if $G$ is $k$-connected with minimum
degree $\delta$, $X$ is a longest cycle in $G$, and $G-X$ contains
a path of length at least $k-2$, then $|X|\geq k(\delta-k+2)$.
The present authors~\cite{ma2026longest} recently proved this conjecture
under the additional assumption $\delta\geq6k$.
The result of~\cite{ma2026longest} is sufficient to complete the proof of Theorem~\ref{thm:smith-linear}.
However, to keep the argument self-contained, we give a shorter proof of the following Dirac-type bound, without attempting to optimize the constants.

\begin{thm}\label{thm:small-circ}
For every integer $k\geq2$, let $G$ be a $k$-connected graph with minimum degree $\delta$ and let $X$ be a longest cycle in $G$. Set $p:=p(G-X)$. Then $|X|>(\min\{k,p\}-2)\cdot \delta/25$.
\end{thm}

Throughout the rest of this section, we fix the following conventions. 
Let $G,X,\delta,p$
be as in its statement. If $\min\{k,p\}\leq2$, then the required
bound follows from $|X|>0$. We may therefore assume that
$\min\{k,p\}\geq3$.
Let $G_X:=G-X$, and let $P=u_0u_1\cdots u_p$ be a longest path in $G_X$.
Let $H$ be the component of $G_X$ that contains $P$.
Let $T$ be a subtree of $H$ of the form
$T=P\cup\bigcup_{i=0}^pP_i$, where $P_0,\ldots,P_p$ are pairwise
disjoint paths in $G_X$ and, for each $0\leq i\leq p$, the path $P_i$ connects
$u_i$ to a vertex $v_i \in V(G_X)$ and satisfies
$V(P_i)\cap V(P)=\{u_i\}$. Each $P_i$ may be trivial. Among all such
trees, choose $T$ so that
\begin{itemize}
    \item[(1)] $|V(T)|$ is as large as possible;
    \item[(2)] subject to (1), the number of indices $i$ for which
    $u_i=v_i$ is as small as possible.
\end{itemize}

For a vertex $v$ and a set $U\subseteq V(G)$, let
$N_G(v,U):=N_G(v)\cap U$ and $d_G(v,U):=|N_G(v,U)|$. For
$U\subseteq V(G)$, let
$N_G(U):=(\bigcup_{u\in U}N_G(u))\setminus U$. If $H$ is a subgraph of
$G$, we abbreviate $N_G(v,V(H))$, $d_G(v,V(H))$, and $N_G(V(H))$ to
$N_G(v,H)$, $d_G(v,H)$, and $N_G(H)$, respectively.

The specific choice of $T$ has the following properties.
\begin{lemma}\label{lem:tree-properties}
    For the subtree $T$ chosen above and every $0\leq i\leq p$, the following statements hold.
    \begin{itemize}
        \item[(1)] Every neighbor of $v_i$ in $H$ belongs to $T$, and
        $\deg_T(v_i)\leq2$.
        \item[(2)] If $u_i=v_i$, then
        $d_G(v_i,P_j)\leq2$ for every $j\neq i$.
    \end{itemize}
\end{lemma}

\begin{proof}
    If $v_i$ has a neighbor in $H-T$, then adding this neighbor to
    $P_i$ would increase $|V(T)|$, a contradiction. Hence every neighbor of $v_i$ in $H$ belongs to $T$.
    If $v_i\notin V(P)$, then $v_i$ is a leaf of $T$ and $\deg_T(v_i)=1$. Otherwise, $v_i=u_i$ and $P_i$ is trivial, so $\deg_T(v_i)=\deg_P(u_i)\leq2$. This proves (1).

    For (2), suppose for contradiction that $u_i=v_i$ and
    $d_G(v_i,P_j)\geq3$ for some $j\neq i$.
    Then there exists a vertex $w\in N_G(v_i,P_j)$ such that
    $|P_j[u_j,w]|\geq2$.
    Let $w^-$ be the immediate predecessor of $w$ on $P_j$ in the
    linear order from $u_j$ to $v_j$. Replace the pair $(P_j,P_i)$ by
    \[
      \bigl(P_j[u_j,w^-],\ \{v_iw\}\cup P_j[w,v_j]\bigr).
    \]
    These two paths are disjoint, and intersect $P$
    only at $u_j$ and $u_i=v_i$, respectively.
    Since $|P_j[u_j,w]|\geq2$, both of them are nontrivial. This replacement therefore gives another tree with the same vertex set as
    $T$ but fewer indices $h$ satisfying $u_h=v_h$, contradicting
    the choice of $T$.
\end{proof}
For a block of $H$, we introduce the following parameter, which is also called \textit{codiameter} in~\cite{Fan1990}.
\begin{definition}[$f(B)$]\label{def:fb}
For a block $B$ of $H$ with at least two vertices, define $f(B)$ to be the
largest integer such that, for every pair of distinct vertices of $B$, there
is a path in $B$ between them of length at least $f(B)$. Set $f(B):=0$ when $B$ is
an isolated vertex.
\end{definition}

The next lemma bounds this parameter using the neighbors of $v_i$ in
the part of $T$ contained in the block.

\begin{lemma}\label{lem:fB}
For every block $B$ in $H$ and every $0\leq i\leq p$, we have
$f(B)\geq d_G(v_i,V(B)\cap V(T))/8$.
\end{lemma}

\begin{proof}
Fix a block $B$ and an index $0\leq i\leq p$, and set
$d:=d_G(v_i,V(B)\cap V(T))$.
If $B$ is an isolated vertex, then $d=0$. Suppose that $B$ is not an
isolated vertex. Since $f(B)\geq1$, the result is immediate when
$d\leq8$. We may therefore assume that $d>8$. In particular, $B$ has
at least three vertices and is 2-connected.

We claim that $B$ contains a cycle of length at least $d/4$. We first
show that this claim implies the lemma. Let $C$ be such a cycle, and
let $x,y$ be distinct vertices of $B$. Then Menger's theorem gives two
disjoint $(\{x,y\},C)$-paths in $B$. These two paths, together with the longer subpath of $C$ between their endpoints on $C$, form an $(x,y)$-path in $B$ of length at least
$|C|/2\geq d/8$. Thus $f(B)\geq d/8$.

It remains to prove the claim. If $d\leq12$, any cycle in $B$ has
length at least $3\geq d/4$. Hence assume that $d>12$. By Lemma~\ref{lem:block-convexity}(1), a vertex outside
the block $B$ has at most one neighbor in $B$.
In particular, $v_i\in V(B)$.
By Lemma~\ref{lem:block-convexity}(1) again,
the graph $T_1:=T[V(B)\cap V(T)]$ is a subtree of $T$ containing $v_i$, and $d=d_G(v_i,T_1)$.

Suppose first that $T_1\subseteq P_i$. Choose
$a\in N_G(v_i,T_1)$ maximizing $|T_1[a,v_i]|$. Then $|T_1[a,v_i]|\geq d$. In particular, we have 
$v_ia\notin E(T_1)$ since $d>1$. Since $B$ is induced in $H$, $T_1[a,v_i]\cup\{v_ia\}$ is a cycle in $B$ of length greater
than $d$.
We may now assume that $T_1\nsubseteq P_i$. Since $T_1$
is connected and contains $v_i$, we have $P_i\subseteq T_1$. The graph
$T_1-P_i$ has at most two components. Let $T_2$ be the subtree of
$T_1$ induced by the union of $V(P_i)$ and the vertex set of a component of
$T_1-P_i$ that
contains at least half of the neighbors of $v_i$ in
$V(T_1-P_i)$. Then
\[
 d_G(v_i,T_2)\geq d_G(v_i,P_i)+
 \frac{d-d_G(v_i,P_i)}{2}\geq\frac d2>6.
\]
Let $\alpha_1,\ldots,\alpha_{r-1}$ be all indices $j\neq i$ satisfying
$N_G(v_i,V(P_j)\cap V(T_2))\neq\emptyset$, ordered so that
$u_{\alpha_1},\ldots,u_{\alpha_{r-1}},u_i$ occur in this order along
$T_2\cap P$. Let $\alpha_r:=i$, and, for each $j\in[r]$, let
$d_j:=d_G(v_i,V(P_{\alpha_j})\cap V(T_2))$. Then
$\sum_{j=1}^r d_j=d_G(v_i,T_2)\geq d/2$.

Suppose first that $u_i=v_i$. Then $d_r=0$, and
Lemma~\ref{lem:tree-properties}(2) gives
\begin{align}\label{eq:length1}
    d_G(v_i,T_2)=\sum_{j=1}^{r-1}d_j\leq2(r-1).
\end{align}
In particular, $r\geq5$. Choose
$x\in N_G(v_i,V(P_{\alpha_1})\cap V(T_2))$. The vertices
$u_{\alpha_1},\ldots,u_{\alpha_{r-1}},u_i$ are distinct, and hence
$|P[u_{\alpha_1},u_i]|\geq r-1\geq4$. This implies 
$xv_i\notin E(P)$, and thus
\[
 C:=P[u_{\alpha_1},u_i]\cup
 P_{\alpha_1}[u_{\alpha_1},x]\cup\{xv_i\}
\]
is a cycle in $B$. By~\eqref{eq:length1}, 
\[
 |C|\geq |P[u_{\alpha_1},u_i]|\geq  r-1\geq\frac{d_G(v_i,T_2)}2\geq\frac d4.
\]

It remains to consider $u_i\neq v_i$. If $r=1$, all neighbors of
$v_i$ in $T_2$ lie on $P_i$. Choose
$a\in N_G(v_i,P_i)$ maximizing $|P_i[a,v_i]|$. Then
$|P_i[a,v_i]|\geq d_G(v_i,T_2)$. Therefore
$P_i[a,v_i]\cup\{v_ia\}$ is a cycle in $B$ of length greater than
$d_G(v_i,T_2)\geq d/2$.
Now suppose that $r\geq2$. Since $P_i$ is nontrivial and
$P_i\subseteq T_2$, we have $d_r\geq1$. For each $j\in[r]$, choose
$x_j\in N_G(v_i,V(P_{\alpha_j})\cap V(T_2))$ maximizing
$|P_{\alpha_j}[u_{\alpha_j},x_j]|$. Then
\[
 |P_{\alpha_j}[u_{\alpha_j},x_j]|\geq d_j-1.
\]
For each $j\in[r-1]$,
\[
 P_{\alpha_j}[u_{\alpha_j},x_j]
 \cup\{x_jv_i,v_ix_{j+1}\}
 \cup P_{\alpha_{j+1}}[x_{j+1},u_{\alpha_{j+1}}]
\]
is a $(u_{\alpha_j},u_{\alpha_{j+1}})$-path in $H$ whose internal
vertices are disjoint from $P$ and whose length is at least
$d_j+d_{j+1}$. Since $P$ is a longest path in $G-X$, this path has length
at most $|P[u_{\alpha_j},u_{\alpha_{j+1}}]|$. Hence
\[
 |P[u_{\alpha_j},u_{\alpha_{j+1}}]|\geq d_j+d_{j+1}.
\]
Summing this for $j\in [r-1]$ gives
\[
 |P[u_{\alpha_1},u_i]|
 \geq\sum_{j=1}^{r-1}(d_j+d_{j+1})
 \geq\sum_{j=1}^r d_j
 =d_G(v_i,T_2)\geq\frac d2.
\]
Finally,
\[
 P[u_{\alpha_1},u_i]\cup P_i[u_i,v_i]
 \cup\{v_ix_1\}\cup P_{\alpha_1}[x_1,u_{\alpha_1}]
\]
is a cycle in $B$, and its length is at least
$|P[u_{\alpha_1},u_i]|\geq d/2$. This proves the claim and the lemma.
\end{proof}

We also use the following observation, which appeared
in~\cite[Lemma~3.1]{ma2026longest}; we include its short proof.

\begin{lemma}\label{lem:path-sets}
If $A,B$ are nonempty vertex sets on a path $R$, then some $a\in A$ and
$b\in B$ satisfy $|R[a,b]|\geq (|A|+|B|-2)/2$.
\end{lemma}

\begin{proof}
Let $a_1,a_2$ and $b_1,b_2$ be respectively the first and last vertices
of $A$ and $B$ in an orientation of $R$. Then
\[
 |R[a_1,b_2]|+|R[b_1,a_2]|
 \geq |R[a_1,a_2]|+|R[b_1,b_2]|
 \geq |A|+|B|-2,
\]
so one of the two paths on the left has the required length.
\end{proof}

\begin{proof}[\normalfont\textbf{Proof of Theorem~\ref{thm:small-circ}}]
Retain the notation introduced above: $P=u_0\cdots u_p$ is
a longest path in $G-X$, $H$ is the component containing $P$, and $T$
and the paths $P_i$ are chosen as above, with each $P_i$ having endpoint $v_i$.

Write $Z:=\{v_0,\dots,v_p\}$. We first prove
\begin{equation}\label{eq:terminal-X-edges}
|X| > \sum_{v\in Z} d_G(v,X).
\end{equation}
Let $\{x_j:j\in[r]\}$ be the set of vertices $x\in V(X)$ for which
$N_G(x)\cap Z$ is nonempty, listed in cyclic order on $X$. If this
set is empty, \eqref{eq:terminal-X-edges} is immediate as the right-hand side is $0$.
We thus assume
$r\geq1$. For $j\in[r]$, let $A_j:=N_G(x_j)\cap Z$.
    
    First assume $r=1$. Then
    $\sum_{v\in Z}d_G(v,X)=|A_1|$.
    If $|A_1|\leq2$, then $|X|\geq3>|A_1|$, so~\eqref{eq:terminal-X-edges}
    follows. Hence assume $|A_1|\geq3$.
    Let $\alpha$ and $\beta$ be the smallest and largest elements of the index set $\{i:v_i\in A_1\}$, respectively.
    Then $\{x_1v_{\alpha}\}\cup P_{\alpha}\cup
    P[u_{\alpha},u_{\beta}]\cup P_{\beta}\cup\{v_{\beta}x_1\}$ is a
    cycle in $G$ with length at least
    $|P[u_{\alpha},u_{\beta}]|+2\geq |A_1|+1$.
    The maximality of $|X|$ then gives
    $|X|\geq |A_1|+1>\sum_{v\in Z}d_G(v,X)$.
    
    Now assume $r\geq 2$.
    Let $A_{r+1}:=A_1$ and $x_{r+1}:=x_1$. For each $j\in[r]$, applying
    Lemma~\ref{lem:path-sets} to $\{u_i:v_i\in A_j\}$ and
    $\{u_i:v_i\in A_{j+1}\}$ gives indices $a,b$ with
    $0\le a,b\le p$, $v_a\in A_j$, and
    $v_b\in A_{j+1}$ such that
    \[
        |P[u_a,u_b]|\geq
        \frac{|\{u_i:v_i\in A_j\}|+|\{u_i:v_i\in A_{j+1}\}|-2}{2}
        =\frac{|A_j|+|A_{j+1}|-2}{2}.
    \]
    If $|A_j|+|A_{j+1}|>2$, then the right-hand side is positive, so $a\neq b$. Hence
    \[
    \{x_jv_a\}\cup P_a[v_a,u_a]\cup P[u_a,u_b]\cup P_b[u_b,v_b]\cup \{v_bx_{j+1}\}
    \]
    is an $(x_j,x_{j+1})$-path that meets $X$ only at $\{x_j,x_{j+1}\}$ and has length at least $2+(|A_j|+|A_{j+1}|-2)/2$.
    If $|A_j|=|A_{j+1}|=1$, then $G$ still contains an $(x_j,x_{j+1})$-path whose internal vertices lie outside $X$ and whose length is at least $2$: if $a=b$, use $x_jv_ax_{j+1}$; otherwise, use the path displayed above. In both cases, the maximality of $|X|$ gives
    $\operatorname{dist}_X(x_j,x_{j+1})\ge 2+(|A_j|+|A_{j+1}|-2)/2$.
    Summing over $j$, using $A_{r+1}=A_1$, we obtain
    \begin{align*}
        |X| &\geq \sum_{j=1}^r \operatorname{dist}_X(x_j, x_{j+1})
        \geq \sum_{j=1}^r \left(2+\frac{|A_j|+|A_{j+1}|-2}{2}\right)
        > \sum_{j=1}^r |A_j| = \sum_{i=0}^p d_G(v_i,X).
    \end{align*}
    This proves~\eqref{eq:terminal-X-edges}.
    
    By Lemma~\ref{lem:tree-properties}(1), all neighbors of $v_i$ in $H$ belong to
    $T$. Hence
    $d_G(v_i,T)+d_G(v_i,X)=d_G(v_i)\geq\delta$ for every
    $0\leq i\leq p$.
    Therefore,
    \[
    \sum_{i=0}^p d_G(v_i,T)
      \geq(p+1)\delta-\sum_{v\in Z}d_G(v,X)
      >(p+1)\delta-|X|,
    \]
    where the strict inequality follows from~\eqref{eq:terminal-X-edges}.
    Set $w:=\lfloor\min\{k,p\}/3\rfloor>0$. 
    Let $W\subseteq Z$ consist of the $w$ vertices with the largest
    values of $d_G(v_i,T)$. Then
    \begin{align}\label{eq:sum-W}
        \sum_{v\in W}d_G(v,T)
      \geq\frac{w}{p+1}\sum_{i=0}^p d_G(v_i,T)
      >w\delta-\frac{|X|}{3}.
    \end{align}

    Let $\mathcal B$ be the family of blocks of $H$ containing an edge of $T$ incident with a vertex of $W$.
    For every $v\in W$, each edge of $H$ incident with $v$
    belongs to a block in $\mathcal B$.
    Indeed, by Lemma~\ref{lem:tree-properties}(1), every neighbor $u$ of $v$ in $H$ lies in $T$. If $uv\in E(T)$, the assertion follows from the definition of $\mathcal B$.
    Otherwise, the cycle $\{uv\}\cup T[u,v]$ contains an edge of $T$ incident with $v$, so the block containing this cycle belongs to $\mathcal B$.
    Since each edge belongs to a unique block, the sum of
    $d_G(v,V(B)\cap V(T))$ over all $B\in\mathcal B$ containing $v$ equals $d_G(v,T)$.
    Summing these equalities over $v\in W$ and using~\eqref{eq:sum-W} gives
    \begin{equation}\label{eq:block-degree-mass}
    \sum_{v\in W}\sum_{\substack{B\in\mathcal B\\V(B)\ni v}}
    d_G(v,V(B)\cap V(T))
    =\sum_{v\in W}d_G(v,T)
    >w\delta-\frac{|X|}{3}.
    \end{equation}

    Let $S$ be a set that consists of all vertices in $W$ together with every vertex adjacent to
    a vertex of $W$ in $T$. By Lemma~\ref{lem:tree-properties}(1),
    $|S|\leq3|W|=3w\leq\min\{k,p\}$. 
    By definition of $\mathcal B$, each $B\in\mathcal B$ contains at least one edge incident to $W$.
    Hence, every $B\in\mathcal B$ has $|S\cap V(B)|\geq2$.
    We now prove
    \begin{equation}\label{eq:block-path-bound}
       |X|>\sum_{B\in\mathcal B}\big|S\cap V(B)\big|\cdot f(B).
    \end{equation}
    By Theorem~\ref{thm:dirac},
    $|X|=c(G)\geq k\geq|S|$. Since $G$ is $k$-connected, Menger's
    theorem gives $|S|$ disjoint $(S,X)$-paths
    $\{L_i\}_{i\in[|S|]}$, where $L_i$ has endpoints $s_i\in S$ and
    $y_i\in V(X)$, and $S$ consists precisely of these $s_i$.
    We may assume that $y_1, \dots, y_{|S|}$ appear in cyclic order along $X$.
    For each $i\in[|S|]$, let $z_i$ be the vertex adjacent to $y_i$ in
    $L_i$. Since the paths $L_i$ are disjoint, the vertices
    $z_1,\dots,z_{|S|}$ are distinct and belong to $H$.

    For each $i\in[|S|]$, let $t_i$ be the vertex of $T(H)$ equal to
    $z_i$ when $z_i$ is a cut-vertex of $H$, and equal to the unique
    block containing $z_i$ otherwise. 
    Assume all indices are taken modulo $|S|$.
    For each $h\in[|S|]$, let
    $\mathcal B_h$ be the family of blocks in $\mathcal B$ on
    $T(H)[t_h,t_{h+1}]$.
    \begin{samepage}
    As shown in Figure~\ref{fig:block-paths}, the blocks of $H$ are represented by ellipses on the left and by white vertices of $T(H)$ on the right.
    The shaded ellipses denote the blocks in $\mathcal B$.
    The cut-vertices of $H$ are represented by black dots where ellipses meet on the left and by black vertices on the right.
    In the illustrated example, $\mathcal B_i=\{B_1,B_2,B_3\}$
    and $\mathcal B_{i+1}=\{B_2,B_3\}$.
    \end{samepage}

\begin{figure}[!t]
\centering
\begin{tikzpicture}[
    x=0.76cm,
    y=0.82cm,
    line cap=round,
    line join=round,
    block/.style={
        draw=black,
        fill=white,
        line width=0.75pt,
        ellipse,
        minimum width=1.60cm,
        minimum height=1.02cm,
        inner sep=0pt
    },
    selected block/.style={
        block,
        pattern=north east lines,
        pattern color=black
    },
    block node/.style={
        circle,
        draw=black,
        fill=white,
        line width=0.55pt,
        minimum size=4.0pt,
        inner sep=0pt
    },
    cut vertex/.style={
        circle,
        draw=black,
        fill=black,
        line width=0.55pt,
        minimum size=4.0pt,
        inner sep=0pt
    },
    endpoint/.style={
        circle,
        draw=black,
        fill=black,
        minimum size=4.0pt,
        inner sep=0pt
    }
]

\begin{scope}
    \node[block]          at (0.80,0) {};
    \node[selected block] at (2.90,0) {};
    \node[block]          at (5.00,0) {};
    \node[selected block] at (7.10,0) {};
    \node[selected block] at (9.20,0) {};
    \node[block]          at (11.30,0) {};

    \node[block] at (5.00, 1.24) {};
    \node[selected block] at (0.80, 1.24) {};
    \node[selected block] at (7.10,-1.24) {};
    \node[selected block] at (11.30,1.24) {};

    \node[font=\scriptsize,fill=white,inner sep=0.7pt,yshift=8mm]
        at (2.90,-0.12) {\(B_1\)};
    \node[font=\scriptsize,fill=white,inner sep=0.7pt,yshift=8mm]
        at (7.10,-0.12) {\(B_2\)};
    \node[font=\scriptsize,fill=white,inner sep=0.7pt,yshift=8mm]
        at (9.20,-0.12) {\(B_3\)};

    \node[cut vertex] at (1.85,0) {};
    \node[cut vertex] at (3.95,0) {};
    \node[cut vertex] at (6.05,0) {};
    \node[cut vertex] at (8.15,0) {};
    \node[cut vertex] at (10.25,0) {};
    \node[cut vertex] at (0.80, 0.62) {};
    \node[cut vertex] at (5.00, 0.62) {};
    \node[cut vertex] at (7.10,-0.62) {};
    \node[cut vertex] at (11.30,0.62) {};

    \node[font=\scriptsize,anchor=south east,yshift=-0.7mm]
        at (1.79,0.08) {\(z_i\)};
    \node[font=\scriptsize,anchor=south,yshift=-0.7mm,xshift=4mm]
        at (3.95,0.08) {\(z_{i+2}\)};

    \node[endpoint] at (11.52,0) {};
    \node[font=\scriptsize,anchor=south,yshift=-0.7mm]
        at (11.52,0.08) {\(z_{i+1}\)};
\end{scope}

\begin{scope}[xshift=10.35cm]
    \draw[line width=0.65pt]
        (0.00,0) -- (0.75,0) -- (1.50,0) -- (2.25,0) --
        (3.00,0) -- (3.75,0) -- (4.50,0) -- (5.25,0) --
        (6.00,0) -- (6.75,0) -- (7.50,0);
    \draw[line width=0.65pt]
        (0.00,0) -- (0.00,0.86) -- (0.00,1.72);
    \draw[line width=0.65pt]
        (3.00,0) -- (3.00,0.86) -- (3.00,1.72);
    \draw[line width=0.65pt]
        (4.50,0) -- (4.50,-0.86) -- (4.50,-1.72);
    \draw[line width=0.65pt]
        (7.50,0) -- (7.50,0.86) -- (7.50,1.72);

    \foreach \x in {0.00,1.50,3.00,4.50,6.00,7.50}
        \node[block node] at (\x,0) {};
    \node[block node] at (0.00, 1.72) {};
    \node[block node] at (3.00, 1.72) {};
    \node[block node] at (4.50,-1.72) {};
    \node[block node] at (7.50, 1.72) {};

    \foreach \x in {0.75,2.25,3.75,5.25,6.75}
        \node[cut vertex] at (\x,0) {};
    \node[cut vertex] at (0.00, 0.86) {};
    \node[cut vertex] at (3.00, 0.86) {};
    \node[cut vertex] at (4.50,-0.86) {};
    \node[cut vertex] at (7.50, 0.86) {};

    \node[font=\scriptsize,anchor=south] at (1.50,0.16)
        {\(B_1\)};
    \node[font=\scriptsize,anchor=south] at (4.50,0.16)
        {\(B_2\)};
    \node[font=\scriptsize,anchor=south] at (6.00,0.16)
        {\(B_3\)};
\node[font=\scriptsize,anchor=north] at (0.75,-0.14) {\(t_i\)};
\node[font=\scriptsize,anchor=north] at (2.25,-0.14) {\(t_{i+2}\)};
\node[font=\scriptsize,anchor=north] at (7.50,-0.14) {\(t_{i+1}\)};

\end{scope}
\end{tikzpicture}
\caption{The blocks of $H$ and the corresponding block-cut tree.}
\label{fig:block-paths}
\end{figure}
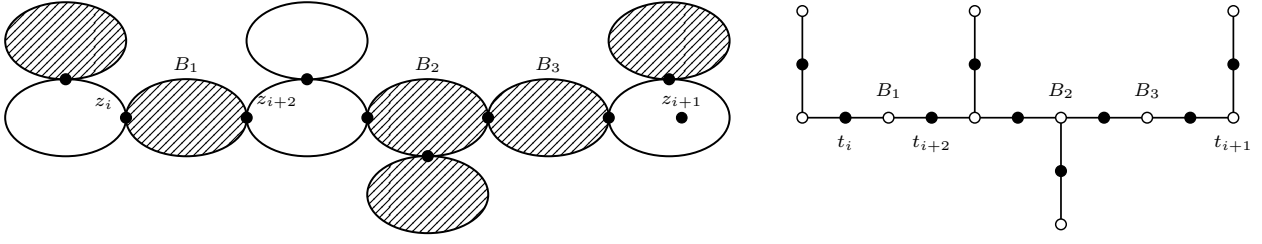

    Fix $h\in[|S|]$ and a longest $(z_h,z_{h+1})$-path $R$ in $H$. 
    Then $R$ contains an edge of every block on $T(H)[t_h,t_{h+1}]$, and in particular of every block in $\mathcal{B}_h$.
    By Lemma~\ref{lem:block-convexity}(1), the edges of $R$ in each $B\in\mathcal B_h$ form a nontrivial subpath of $R$.
    By Definition~\ref{def:fb} and the maximality of $|R|$, this subpath has length at least $f(B)$.
    Moreover, $y_hz_h\cup R\cup z_{h+1}y_{h+1}$ is a $(y_h,y_{h+1})$-path in $G$ whose internal vertices lie outside $X$, so the maximality of $|X|$ gives
    \[
       \operatorname{dist}_X(y_h,y_{h+1})> |R|\geq
       \sum_{B\in\mathcal B_h}f(B).
    \]
    Summing over $h\in[|S|]$ gives
    \begin{equation}\label{eq:block-family-sum}
       |X|\geq\sum_{h=1}^{|S|}
       \operatorname{dist}_X(y_h,y_{h+1})
       >\sum_{h=1}^{|S|}\sum_{B\in\mathcal B_h}f(B).
    \end{equation}

    To prove~\eqref{eq:block-path-bound}, we consider how many times a fixed $B\in \mathcal{B}$ is counted in the right-hand side of~\eqref{eq:block-family-sum}.
    We claim that for a fixed $B\in\mathcal B$, if $s_i,s_j\in V(B)$ for distinct indices $i$ and $j$, then $B\in\mathcal B_h$ for some index $h$ in the interval $[i,j)$ along the cyclic
    order $1,\ldots,|S|,1$. 
    Since $T(H)$ is a tree, $T(H)[t_i,t_j]$ is contained in
    $\bigcup_{h\in[i,j)}T(H)[t_h,t_{h+1}]$,
    where $[i,j)$ is the cyclic interval specified above.
    Since $B\in\mathcal B$, it suffices to show that
    $B$ lies on $T(H)[t_i,t_j]$.
    Suppose not. Then $t_i$ and $t_j$ lie in the same component of $T(H)-B$.
    However, $L_i[s_i,z_i]$ and $L_j[s_j,z_j]$ are disjoint paths in $H$,
    each with one endpoint in $B$, while their other endpoints $z_i,z_j$
    correspond to $t_i,t_j$, respectively. Hence both paths must contain
    the same cut-vertex of $H$ in $B$, a contradiction.
    This proves the claim.
    For a fixed $B\in\mathcal B$, apply the claim to consecutive indices $i$ with $s_i\in V(B)$, in their cyclic order.
    These $|S\cap V(B)|$ cyclic intervals are pairwise disjoint, and each contains an index $h$ for which $B\in\mathcal B_h$.
    Thus there are at least $|S\cap V(B)|$ such indices $h$.
    Together with
    \eqref{eq:block-family-sum}, this proves
    \eqref{eq:block-path-bound}.

    Since $W\subseteq S$, Lemma~\ref{lem:fB} gives, for every
    $B\in\mathcal B$,
    \[
    |S\cap V(B)|\cdot f(B) \geq |W\cap V(B)|\cdot f(B)
      \geq\frac18\sum_{v\in W\cap V(B)}
         d_G(v,V(B)\cap V(T)).
    \]
    Summing this inequality over $B\in \mathcal{B}$ and using~\eqref{eq:block-degree-mass} and~
    \eqref{eq:block-path-bound}, we obtain
    \[
       |X|>\sum_{B\in\mathcal B}\big|S\cap V(B)\big|\cdot f(B)\geq\frac18\sum_{v\in W}
          \sum_{\substack{B\in\mathcal B\\V(B)\ni v}}
          d_G(v,V(B)\cap V(T))
       >\frac{w\delta-|X|/3}{8}.
    \]
    Rearranging this gives $|X|>3w\delta/25\geq(\min\{k,p\}-2)\delta/25$. This proves Theorem~\ref{thm:small-circ}.
\end{proof}

\section{Completing the proof of Theorem~\ref{thm:smith-linear}}\label{sec:complete}
Now we complete the proof of our main theorem.

\begin{proof}[\normalfont\textbf{Proof of Theorem~\ref{thm:smith-linear}}]
Let $X$ and $Y$ be two longest cycles in a $k$-connected graph $G$, and
let $m:=|V(X)\cap V(Y)|$. Since $k\geq2$, two longest cycles in $G$
meet in at least two vertices, so $m\geq2$. If $c(G)>23mk$, then
Lemma~\ref{lem:large-circ} gives $m\geq k/600$.
We may therefore assume that $c(G)\leq23mk$.

Suppose for contradiction that $k>600m$.
Let $p:=p(G-X)$. Since the at most $m$ nontrivial maximal paths of
$Y-X$ contain at least $c(G)-2m$ edges in total, we have
$p\geq c(G)/m-2$.
We verify that both $k-2$ and $p-2$ are greater than $25c(G)/k$. Since $k$ is an integer and $k>600m$,
\[
 k-2\geq600m-1>575m\geq\frac{25c(G)}k.
\]
Moreover, using $c(G)\geq k$ (which can be derived from Theorem~\ref{thm:dirac}), we have
\[
p-2\geq\frac{c(G)}m-4
\geq\left(\frac{k}{m}-4\right)\frac{c(G)}k
>\frac{25c(G)}k.
\]
The above two inequalities together imply that
$\min\{k,p\}-2>25c(G)/k$.
Applying Theorem~\ref{thm:small-circ} and using $\delta\geq k$, we obtain
\[
 c(G)=|X|>\frac{(\min\{k,p\}-2)\cdot \delta}{25}\geq \frac{(\min\{k,p\}-2)\cdot k}{25}>c(G),
\]
a contradiction. This completes the proof of Theorem~\ref{thm:smith-linear}.
\end{proof}

\section{Concluding remarks}\label{sec:conclude}
In this paper, we prove that any two longest cycles in a $k$-connected graph share at least $k/600$ vertices, establishing Smith's conjecture up to a constant factor. 
The resulting constant $1/600$ has not been optimized. The best constant we are aware of is close to $1/20$.
We conclude with several applications and related problems.

\subsection{A linear bound for Hippchen's conjecture}
Our proof also applies to the following conjecture of Hippchen~\cite{Hippchen2008}, which is an analogue of Smith's conjecture for longest paths.

\begin{conjecture}[Hippchen's conjecture~\cite{Hippchen2008}]
\label{conj:Hippchen}
For every $k\geq 2$, any two longest paths in a $k$-connected graph
have at least $k$ common vertices.
\end{conjecture}

Recall that $p(G)$ denotes the length of a longest path in $G$. 
We prove a linear-bound version of Conjecture~\ref{conj:Hippchen} by combining Theorem~\ref{thm:smith-linear} with the standard reduction obtained by adding a universal vertex (cf.~\cite{gutierrez2024two}).

\begin{thm}
\label{thm:longest-path-intersection}
For every integer $k\geq 2$, any two longest paths in a
$k$-connected graph share at least $\left\lfloor\frac{k}{600}\right\rfloor$ common vertices.
\end{thm}

\begin{proof}
Let $G^+$ be obtained from $G$ by adding a new vertex $z$ adjacent to
every vertex of $G$. 
Then $G^+$ is a $(k+1)$-connected graph with $c(G^+)=p(G)+2$.
Consider any two longest paths $P,Q$ in $G$.
Then $zPz$ and $zQz$ are two longest cycles of $G^+$. By Theorem~\ref{thm:smith-linear}, we have
$$|V(P)\cap V(Q)|
      =|V(zPz)\cap V(zQz)|-1
      \geq \left\lceil\frac{k+1}{600}\right\rceil-1
      =\left\lfloor\frac{k}{600}\right\rfloor.$$
This proves Theorem~\ref{thm:longest-path-intersection}.
\end{proof}

\noindent\textbf{Remark.}
The same reduction shows that Smith's conjecture for $(k+1)$-connected
graphs would imply Hippchen's conjecture for $k$-connected graphs.

\subsection{Implications for vertex-transitive graphs}
Using Theorem~\ref{thm:smith-linear}, we can derive a lower bound on the length of longest cycles in vertex-transitive graphs.

There are two famous conjectures concerning vertex-transitive graphs.
Lov\'asz's conjecture~\cite{lovasz1969combinatorial} asserts that every
connected vertex-transitive graph has a Hamiltonian path, while
Thomassen's conjecture (see~\cite{barat2010his}) asserts that, apart
from finitely many exceptions, every connected vertex-transitive graph
has a Hamiltonian cycle. 
Both conjectures remain wide open in general, but sufficiently high
degree is known to guarantee Hamiltonicity. Christofides, Hladk\'y,
and M\'ath\'e~\cite{christofides2014hamilton} proved that, for every
fixed $\alpha>0$, every sufficiently large connected vertex-transitive
graph on $n$ vertices with degree at least $\alpha n$ contains a
Hamiltonian cycle. More recently, Bedert,
Dragani\'c, M\"uyesser, and Pavez-Sign\'e~\cite{bedert2026lovasz}
showed that, for some absolute constant $\varepsilon>0$, every sufficiently large connected Cayley graph on $n$ vertices with degree greater
than $n^{1-\varepsilon}$ contains a
Hamiltonian cycle.

On the other hand, substantial progress has been made in improving lower bounds for general vertex-transitive graphs.
Let $G$ be a connected vertex-transitive graph on $n$ vertices.
A classical result of
Babai~\cite{babai1979long} gives $c(G)=\Omega(n^{1/2})$. 
More than four decades later, DeVos~\cite{devos2023longer} improved this to
$c(G)\geq(1-o(1))n^{3/5}$.
Groenland, Longbrake, Steiner, Turcotte,
and Yepremyan~\cite{groenland2024longest} then obtained
$c(G)=\Omega(n^{13/21})$. Norin, Steiner, Thomass\'e, and
Wollan~\cite{norin2025small} further improved the bound to
$c(G)=\Omega(n^{9/14})$. Recently, Buci\'c, Christoph,
Pokrovskiy, and Steiner~\cite{bucic2026towards} proved the current best
bound $c(G)=\Omega\bigl(n^{2/3-o(1)}\bigr).$

A double-counting lemma of DeVos~\cite{devos2023longer} implies that
if every two longest cycles in a vertex-transitive graph $G$ on $n$
vertices meet in at least $m$ vertices, then $c(G)\geq\sqrt{mn}.$
Since a connected $d$-regular vertex-transitive graph with $d\geq 2$ is $\lceil\frac{2(d+1)}{3}\rceil$-connected (see~\cite{godsil2013algebraic}), Theorem~\ref{thm:smith-linear} implies that every two longest cycles share at least $(d+1)/900$ vertices. This gives the following theorem.

\begin{thm}
\label{thm:vertex-transitive-circumference}
For every integer $d\geq 2$, every connected $d$-regular
vertex-transitive graph $G$ on $n$ vertices satisfies
$c(G)=\Omega\bigl(\sqrt{dn}\bigr).$
\end{thm}

For connected $d$-regular vertex-transitive graphs,
Theorem~\ref{thm:vertex-transitive-circumference} improves the exponent
of $d$ in the previous bounds by Chen, Faudree,
and Gould~\cite[Theorem~4]{chen1998intersections} and Ma and Zhao~\cite[Corollary~1.7]{ma2025intersections}. 
It also improves the bound $n^{2/3-o(1)}$ of Buci\'c,
Christoph, Pokrovskiy, and Steiner~\cite{bucic2026towards} in the range
$d=\Omega(n^{1/3})$.

\subsection{A possible strengthening of Theorem~\ref{thm:smith-linear}}\label{sec:further-problems}
As pointed out in the introduction, all quantitative approaches to Conjecture~\ref{conj:smith} prior to this work rely on Tur\'an-type and supersaturation arguments.
One advantage of these approaches is that they establish the following stronger conclusion: if two longest cycles meet in few vertices, then a relatively small vertex cut separates the two cycles.
Such a conclusion can be used, via an argument of DeVos~\cite{devos2023longer}, to derive lower bounds on $c(G)$ for vertex-transitive graphs $G$.
Motivated by this, we propose the following strengthening of Theorem~\ref{thm:smith-linear}, which was also recently suggested by Buci\'{c}, Christoph, Pokrovskiy, and Steiner~\cite{bucic2026towards}.

\begin{conjecture}\label{conj:local-separation}
There exists a constant $C>0$ such that, for every $2$-connected graph $G$ and any two longest cycles $X$ and $Y$ in $G$, there is a set $S\subseteq V(G)$ with
$|S|\leq C\cdot |V(X)\cap V(Y)|$ that separates $X$ and $Y$ in $G$.
\end{conjecture}

We remark that Conjecture~\ref{conj:local-separation} would imply
$c(G)=\Omega(n^{2/3})$ for every connected vertex-transitive graph $G$
on $n$ vertices. This would remove the $o(1)$ loss in the exponent from the recent bound
$c(G)=\Omega(n^{2/3-o(1)})$ of Buci\'c, Christoph, Pokrovskiy, and
Steiner~\cite{bucic2026towards}.

\subsection{Some further problems}
We would like to mention a question explicitly posed by Bondy and
Locke~\cite{bondy1981relative}: what is the largest constant $\alpha$
such that every $3$-connected graph $G$ satisfies
$c(G)\geq\alpha\cdot p(G)$?
They proved that the optimal constant lies between $2/5$ and $1/2$;
see also the generalization to $k$-connected graphs
in~\cite{Locke1982}.
The exact value of the constant $\alpha$ remains unknown.
In Theorem~\ref{lem:amplifier}, we obtain a lower bound of $1/6$ for amplifiers. 
This constant is not best possible. By adapting the argument of Bondy and Locke~\cite{bondy1981relative}, at the cost of substantially greater technical complexity, the bound can be improved to $2/5$.
Determining the optimal constant in Theorem~\ref{lem:amplifier} might shed some light on the above question of Bondy and Locke.

We conclude this paper with the following problem, attributed to Gallai in the
British Combinatorial Conference problem list
\cite{Cameron1997ResearchProblems}: does every triple of longest paths
in a finite connected graph have a common vertex?

\bigskip

{\noindent \bf AI Declaration.} No AI tools have been used in any way to generate the mathematical ideas of this paper.

\end{document}